\documentclass[11pt]{article}
\usepackage{amsmath,amsthm}
\usepackage{amssymb,latexsym}
\usepackage{mathrsfs}
\usepackage{amscd}
\usepackage{enumerate}
\usepackage{color}
\usepackage{cite}
\usepackage{amssymb}
\usepackage[mathscr]{euscript}
\usepackage{graphicx}
\usepackage{titletoc}
\usepackage{hyperref}
\usepackage{esint}

\usepackage{caption}
\titlecontents{section}[0pt]{\addvspace{1pt}\filright}
              {\contentspush{\thecontentslabel\quad}}
              {}{\titlerule*[8pt]{$\cdot$}\contentspage}

\newtheorem{theorem}{Theorem}[section]
\newtheorem{corollary}[theorem]{Corollary}
\newtheorem{lemma}[theorem]{Lemma}
\newtheorem{proposition}[theorem]{Proposition}
\newtheorem{definition}[theorem]{Definition}

\newtheorem{example}[theorem]{Example}

\theoremstyle{definition} \theoremstyle{remark}
\numberwithin{equation}{section}

\begin{document}

\title{\large{{\bf Gromov Hyperbolic Graphs  Arising From  Iterations}}}

\author{Shi-Lei Kong, Ka-Sing Lau and Xiang-Yang Wang}

\date{}

\maketitle

\abstract{For a contractive iterated function system (IFS), it is known that there is a natural hyperbolic graph structure  (augmented tree) on the symbolic space of the IFS that reflects the relationship among neighboring cells, and its hyperbolic boundary with the Gromov metric is H\"{o}lder equivalent to the attractor $K$  \cite{Ka,LW1,LW3}. This setup was taken up to study the probabilistic potential theory on $K$ \cite{KLW1,KL}, and the bi-Lipschitz equivalence on $K$ \cite {LL}. In this paper, we formulate a broad class of hyperbolic graphs, called {\it expansive hyperbolic graphs}, to capture the most essential properties from the augmented trees and the hyperbolic boundaries (e.g., the special geodesics,  bounded degree property, metric doubling property,  and H\"older equivalence). We also study a new setup of ``weighted" IFS and investigate its connection with the self-similar energy form in the analysis of fractals.}

\tableofcontents

\renewcommand{\thefootnote}{}

\footnote{The research is supported in part by the HKRGC grant, SFB 1283 of the German Research Council \\ and the NSFC (nos. 11971500, 11831007).}
\renewcommand{\thefootnote}{\arabic{footnote}}
\setcounter{footnote}{0}

\vspace{4mm}

\section{Introduction}
\label{sec:1}

\noindent
Let $\{S_i\}_{i=1}^N$ be a contractive iterated function system (IFS) on ${\mathbb R}^d$, and let $K$ be the attractor. It is well-known that the IFS is associated with a finite word space  $ \Sigma^*$ (symbolic space or coding space), and the limit set $\Sigma^\infty$ is used to represent elements in $K$.
With the intention to carry over the probabilistic potential theory to $K$,  Denker and Sato \cite{DS1,DS2} first constructed a special type of Markov chain $\{Z_n\}_{n=0}^\infty$ on $\Sigma^*$ of the Sierpinski gasket (SG), and showed that the Martin boundary of $\{Z_n\}_{n=0}^\infty$ is homeomorphic to the SG. Motivated by this, Kaimanovich \cite{Ka} introduced an ``augmented tree" $(\Sigma^*, E)$ by adding ``horizontal" edges to the coding tree $(\Sigma^*,E_v)$ according to the neighboring cells in each level of the SG. He showed that the graph is hyperbolic in the sense of Gromov \cite{Gr,Wo}, and the SG equipped with the Euclidean metric is H\"older equivalent to the hyperbolic boundary of $(\Sigma^*, E)$.  He also suggested that this approach would also work for more general IFSs, and could provide another tool to study the geometry and analysis  on fractals.

\bigskip

These initiatives were carried out by the authors in a series of papers \cite{LW1, LW2, LW3, JLW, Wa, KLW1, KLW2, KL}. In \cite{LW1,LW3}, we investigated the systems of neighboring cells to general IFSs on which the hyperbolicity of the augmented trees is valid;  it was shown that the hyperbolic boundaries and the attractors  $K$ are H\"{o}lder equivalent (or homeomorphic) under various circumstances. The H\"older equivalence was applied to the bi-Lipschitz classification of the totally disconnected self-similar sets \cite{LL,DLL}. More importantly, this setup was taken up to study the probabilistic potential theory on the self-similar sets  \cite{KLW1}: for an augmented tree $(X,E)$ defined by an IFS on $\mathbb{R}^d$ satisfying the open set condition, we introduced a class of reversible transient random walks on  $(X, E)$.  By identifying the Matin boundary, the hyperbolic boundary and $K$, we obtain an induced Dirichlet form of Gagliardo-type on $K$.
The relevant energy forms and function spaces were further studied in \cite{KL}.

\bigskip

The augmented tree has very rich structure inherited from the iterations and the attractors. In this paper, our first goal is to formulate a broad class of hyperbolic graphs to capture  the most essential concepts, such as the special geodesics, the bounded degree property, and the doubling property of the Gromov metric on the hyperbolic boundaries, which were used extensively in the study of augmented trees. This new setup allows us to carry out the idea of augmentation much further, and beyond IFS. Besides extending the previous study on the identification of the  hyperbolic boundaries with the attractors $K$, we are able to use the hyperbolic techniques developed to construct new metrics on $K$ that are useful in the analysis of fractals.

\bigskip

 Let $(X, E)$ be a locally finite connected graph with a root $\vartheta \in X$.
 We have the decomposition $E = E_v \cup E_h$, the set of vertical edges (which does not necessarily form a tree) and the set of horizontal edges.
Let ${\mathcal J}_m(x)$ denote the set of descendants in the $m$-th generation of $x \in X$. Denote by $d$ and $d_h$ the graph distances on $(X, E)$ and $(X, E_h)$ respectively, and let $|x| := d(\vartheta,x)$.  We call a rooted graph {\it expansive} if for $x, y \in X$ with $|x| = |y|$,
\begin{equation*}
d_h(x,y) >1 \ \Rightarrow \   d_h(u, v)>1,  \quad \forall \   u\in {\mathcal J}_1(x),\  v\in {\mathcal J}_1(y),
\end{equation*}
and call it {\it $(m, k)$-departing} if
\begin{equation*} 
d_h(x,y) >k \ \Rightarrow \ d_h(u,v) > 2k, \quad \forall \ u\in {\mathcal J}_m(x), \ v\in {\mathcal J}_m(y).
\end{equation*}
Intuitively, the two definitions describe the distances of the descendants that are drifted apart compared to the non-neighboring  predecessors.
The expansive property provides a rather simple form of geodesics ({\it convex geodesics}, see Proposition \ref{th2.3}, Figure \ref {fig:2}), and the $(m,k)$-departing property gives more details of how the graph evolves. These two lead to the following criteria for the hyperbolicity (Theorem \ref{th2.11}), which are crucial in the paper.

\begin{theorem} \label{th0.1}
Let $(X,E)$ be an expansive graph. The following assertions are equivalent.

\ \ (i) $(X,E)$ is hyperbolic;

\ (ii) $\exists$ $L<\infty$ such that the lengths of all horizontal geodesics are
 bounded by $L$;

(iii) $(X,E)$ is $(m,k)$-departing for some positive integers $m$ and $k$.

\end{theorem}
\noindent (By a {\it horizontal geodesic}, we mean a geodesic in $(X, E)$ consisting of edges in $E_h$ only.)
We will call such $(X,E)$ an {\it expansive hyperbolic graph}: together with the expansiveness, (ii) gives a clear geometry of the geodesics in $X$; the $(m,k)$-departing property in (iii) will serve as the workhorse in many of the proofs in this study.

\medskip

We use $\partial X$ to denote the hyperbolic boundary of $(X,E)$, which is a compact space with an associated metric $\theta_a$, $a>0$ small ({\it Gromov metric}, see Definition \ref{de2.4}).  By using the $(m,k)$-departing property, we obtain a sharp description of the equivalent rays in $X$ that converge to the same boundary elements (Proposition \ref {th3.1}) which will be used in a number of estimates. One of the main results is (Theorem \ref{th3.6})

\begin{theorem} \label{th0.2} 
Suppose $(X,E)$ is an expansive hyperbolic graph, and has bounded degree (i.e., $\sup_{x\in X}\deg(x) < \infty$).  Then  $(\partial X, \theta_a)$ is a doubling metric space.
\end{theorem}

\medskip
Recall that an {\it augmented tree} of an IFS  with attractor $K$ is based on the tree of the symbolic space, together with the added horizontal edges  that connect neighboring cells in the same level (see Appendix).  This can easily be reformulated on any vertical rooted graph $(X,E_v)$:
let $(M, \rho)$ be a complete metric space, and let ${\mathcal C}_M$ denote the family of nonempty compact subsets of $M$.  We define an {\it index map} $\Phi: X \to {\mathcal C}_M$  that satisfies $\Phi(y) \subset \Phi(x)$ whenever $y \in \mathcal J_1(x)$, and $\bigcap_{i=0}^\infty \Phi(x_i)$ is a singleton for any geodesic ray $[x_i]_i$ from the root $\vartheta$ (see Definition \ref{de4.1}); likewise, we also have an attractor $K$.  This setup  is very general, which includes  all IFSs  (where $\Phi(x)$ is defined as the cell $K_x$), as well as cases that are not from IFS, e.g., refinement systems of sets.

\medskip

With the vertical graph $(X, E_v)$ and the index map $\Phi$, we define
\begin{equation*}
E_h^{(\infty)} := \big\{(x,y) \in X \times X: |x| = |y|, \ x \neq y, \hbox{ and } \Phi(x)\cap \Phi(y)\not = \emptyset \big\}.
\end{equation*}
Let $E^{(\infty)}= E_v \cup E_h^{(\infty)}$, and call $(X, E^{(\infty)})$ an $AI_\infty$-graph  ({\it augmented index graph of type-$(\infty)$}, or {\it intersection type}). In the case that $\{\Phi(x)\}_{x\in X}$ is of exponential type-$(b)$ (i.e., the diameter $|\Phi(x)|_\rho = O(e^{-b|x|})$ for some $b>0$, as $|x| \to \infty$), for some fixed $\gamma >0$, we define a horizontal edge set by
\begin{equation*}
 E^{(b)}_h (= E^{(b)}_h(\gamma)) :=\big\{(x,y) \in X \times X: |x| = |y|, \ x \neq y, \hbox{ and } {\rm dist}_\rho (\Phi(x),\Phi(y)) \leq \gamma e^{-b|x|}\big\},
\end{equation*}
  Let $E^{(b)} = E_v \cup E^{(b)}_h$, and  call  $(X, E^{(b)})$ an $AI_b$-graph ({\it augmented index graph of type-$(b)$}).

 \vspace {0.1cm}
Both $AI_\infty$- and $AI_b$-graphs are expansive.
The condition that defines an $AI_\infty$-graph is more intuitive, and $E_h^{(\infty)}$ consists of fewer edges.
 However, concerning the hyperbolicity, the $AI_b$-graph has the advantage that we do not need to know the fine structure of $K$ a priori.  More precisely, by using Theorem \ref{th0.1}, we show that (Theorem \ref{th4.5})

\medskip

\begin{theorem}  \label{th0.3}
The $AI_b$-graph is  $(m,1)$-departing for some $m\geq 1$, and hence hyperbolic. Moreover,  the index map $\Phi$ induces a bijection $\kappa: \partial X \to K$ that is a H\"older equivalence, i.e., $\rho(\kappa(\xi),\kappa(\eta))^{a/b} \asymp \theta_a(\xi,\eta)$ for all $\xi,\eta \in \partial X$.
\end{theorem}

\medskip

(Here by $f \asymp g$ we mean that there exists $C\geq 1$ such that $C^{-1} f(x) \leq g(x) \le Cf(x)$ for all variables $x$ in a given domain.) The $AI_\infty$-graph is not always hyperbolic (Example \ref{ex6.1}); in order to have that,  we need an additional separation condition $(S_b)$ on $\{\Phi(x)\}_{x \in X}$  (Definition \ref {de5.1}, which is satisfied by  IFS of similitudes with the OSC).

\medskip

\begin{theorem} \label{th0.4}
Suppose $\{\Phi(x)\}_{x \in X}$ is of exponential type-(b), and satisfies condition $(S_b)$ for some $b>0$. Then the $AI_\infty$-graph is hyperbolic, and the induced bijection $\kappa: \partial X \to K$ is H\"older continuous, i.e., $\rho(\kappa(\xi),\kappa(\eta))^{a/b} \leq C \theta_a(\xi,\eta)$ for all $\xi,\eta \in \partial X$.
\end{theorem}

\medskip

The proof of the theorem is in Theorem \ref{th5.4} and Corollary \ref{th4.7}.  We also provide an example (Example \ref {ex6.2}) to show that unlike the $AI_b$-graph, the H\"older continuity of $\kappa$ for the $AI_\infty$-graph cannot be improved to  H\"older equivalence.

\medskip

For the augmented tree of an IFS, the {\it bounded degree} property is important because it allows us to consider certain random walks on it \cite {KLW1,KL}. This property has been characterized in terms of the separation properties such as the open set condition and weak separation condition for  IFSs \cite {LW3, Wa}. In Section \ref{sec:5}, we prove (Theorem \ref{th5.5})

\medskip

\begin{theorem} \label{th0.5}
The $AI_{b}$-graph has bounded degree if and only if  condition ($S_b$) is satisfied.
Also, the $AI_\infty$-graph has bounded degree provided that ($S_b$) is satisfied.
\end{theorem}

In our previous consideration of IFS,  the structure of augmented trees (or $AI_\infty$- and $AI_b$-graphs here) arose from the geometry of $K_x$ under a given metric (usually Euclidean metric). In the analysis of fractals, there are situations that involve weighted IFS,  which give rise to new metrics on $K$ (e.g., the resistance metric in the study of Dirichlet form \cite {Ki1, HW} or the metrics involve in the time change of Brownian motions \cite {Ki2, Ki3, Ki4, GLQR}). This requires new graph structure to  accommodate the new parameter of weights.

 \medskip

In this regard, we let $\{S_j\}_{j=1}^N$ be a contractive IFS on a complete metric space $(M, \rho)$, and let $K$ be the attractor.  Let ${\bf s}= (s_1,\ldots,s_N)$, $s_j \in (0,1)$, be a vector of weights of the maps $S_j$'s.
We regroup the finite words in the symbolic space to form a new  coding tree $(X({\bf s}), E_v)$ (see \eqref{eq6.2}) such that in each level, the $K_x$'s have comparable weights.
In this case the $AI_\infty$-graph is more natural for use (see Section \ref{sec:6}). While we cannot check the hyperbolicity directly (as $\{K_x\}_{x \in X({\bf s})}$ does not satisfy the separation condition $(S_b)$ as in Theorem \ref {th0.4}), we still obtain some rather satisfactory conclusions (Theorem \ref{th6.3}) for the class of p.c.f.~sets \cite {Ki1}.

\begin{theorem} \label{th0.6}
Let $\{S_j\}_{j =1}^N$ be a contractive IFS that has the p.c.f.~property. Then for any weight ${\bf s} \in (0,1)^N$, the $AI_\infty$-graph $(X({\bf s}), E^{(\infty)})$ is an $(m,1)$-departing expansive graph of bounded degree. Consequently the $AI_\infty$-graph is a hyperbolic augmented tree.
\end{theorem}

\medskip

The theorem is applied to study self-similar energy forms and resistance metrics in Section \ref{sec:6}. By using the fact that the natural identification  $\kappa: (\partial X, \theta_a) \to (K, \rho)$ is a homeomorphism,  we can impose a new metric $\tilde \theta_a$ on $K$, 
as a consequence of Theorem \ref{th0.6}.  We  show that {\it if $K$ admits a regular harmonic structure, then $\tilde\theta_a$ is H\"older equivalent to the associated {\it resistance metric} on $K$} (Theorem \ref {th6.7}).

\bigskip

We remark that in \cite{Ki5}, Kigami proposed another construction of metrics through the weighted trees associated with successive partitions on compact metrizable spaces, while he stated that such construction is possible if and only if the ``resolution" graph (which is similar to our $AI_\infty$-graph here) is hyperbolic. He also studied different types of properties among metrics and measures (e.g., Lipschitz equivalence, Ahlfors-regularity, volume doubling, etc.) via the weight functions on trees.

\bigskip

 In a forthcoming paper \cite {KLWa}, by showing that a hyperbolic graph is near-isometric to an expansive hyperbolic graph, we present in greater generality the framework of index maps and augmented index graphs. We can also extend the scope of underlying spaces to quasi-metric spaces  (particularly the  spaces of homogeneous type \cite {Ch, CW}), and study random walks on such hyperbolic graphs.

\bigskip

For the organization of the paper, we introduce the basics of expansive graphs and $(m,k)$-departing property in Section \ref{sec:2}, and prove Theorem \ref{th0.1}.
We study the boundaries of hyperbolic expansive graphs in Section \ref{sec:3}, and prove Theorem \ref{th0.2}.
In Section \ref{sec:4}, we define the index maps, as well as the associated $AI_b$-graphs, $AI_\infty$-graphs, and prove the hyperbolicity of $AI_b$-graphs in Theorem \ref{th0.3}.
For the $AI_\infty$-graphs in Theorems \ref{th0.4} and \ref{th0.5}, we need  some separation properties of the index family, which are detailed in Section \ref{sec:5}.
In Section \ref{sec:6}, we give the two examples as asserted above, and also apply the techniques developed to consider the weighted IFS. An appendix on the augmented tree defined by IFS of similitudes and some related results are included for the convenience of the reader.

\bigskip

\section{Expansive graphs and hyperbolicity}
\label{sec:2}

A {\it graph} $(X,E)$ is a countable set $X$ of vertices together with a set $E$ of edges which is a symmetric subset of $X \times X \setminus \Delta$ ($\Delta:=\{(x,x):x \in X\}$).
It is called {\it locally finite} if for any vertex $x \in X$, $\deg(x):=\# \{y: (x,y) \in E\} < \infty$.
For $x,y \in X$,  we use   $\pi(x, y)$ to denote the {\it geodesic} (path with smallest path length) from $x$ to $y$, and define  the {\it graph distance} $d(x,y)$ by the length of  $\pi(x, y)$  if such path exists  ($d(x,y) = \infty$ otherwise). If $d(x,y)$ is finite for all $x,y \in X$, we say that $(X,E)$ is {\it connected}; in this case $d$ is an integer-valued metric on $X$.

\medskip

In this paper, we assume that $(X,E)$ is a {\it rooted graph}, i.e., a locally finite connected graph in which a vertex $\vartheta \in X$ is fixed as a {\it root}.  We write $|x|:=d(\vartheta,x)$ for $x \in X$, and let $X_n := \{x \in X: |x|=n\}$. Then $X = \bigcup_{n=0}^\infty X_n$. We define a partial order $\prec$ on $X$ with $y \prec x$ if and only if $x$ lies on some $\pi(\vartheta,y)$. For an integer $m \geq 0$ and $x \in X$, let
$$
{\mathcal J}_m(x):= \{y \in X: y \prec x ,\  |y|=|x|+m\}, \quad \mathcal J_{-m}(x):=\{z \in X: x \in \mathcal J_m(z)\}
 $$
be the $m$-th {\it descendant set} and  the $m$-th {\it precedessor set} of $x$  respectively; in general, ${\mathcal J}_m(x)$ is allowed to be empty.
We also write $\mathcal J_*(x):=\{y \in X: y \prec x\}$  and $\mathcal J_{-*}(x):=\{z \in X: x \prec z\}  $ for further use.

\medskip

Let $E_v=\{(x,y) \in E: |x|-|y|=\pm 1\}$ and $E_h=\{(x,y) \in E: |x|=|y|\}$ denote the {\it vertical edge} set and the {\it horizontal edge} set respectively.
Clearly $E=E_v \cup E_h$, and $E_v=\{(x,y) \in X \times X: x \in \mathcal J_1(y) \hbox{ or } y \in \mathcal J_1(x)\}$. We say that a rooted graph $(X,E)$ is {\it vertical} if $E=E_v$. A {\it (rooted) tree} is a vertical rooted graph satisfying $\# \mathcal J_{-1}(x)=1$ for all $x \in X \setminus \{\vartheta\}$.

\medskip

We refer to the {\it horizontal distance} $d_h(\cdot,\cdot)$  as the graph distance on $(X ,E_h)$. We write $x \sim_h y$ for each pair $(x,y) \in E_h \cup \Delta$.
It is clear that $d_h(x, y ) = \infty$ for $|x| \not = |y|$, and $d(x, y) \leq d_h(x, y)$.  In the case that $d(x, y) = d_h(x, y)$, there is a geodesic $\pi(x,y)$ that lies in $(X, E_h)$, called a {\it horizontal geodesic} of $(X, E)$.
\medskip

\begin{definition} \label{de2.1}
We call  $(X, E)$ an {\rm expansive graph} if it is a rooted graph that satisfies for $x,y \in X$,
\begin{equation} \label{eq2.1}
d_h(x,y) >1 \ \Rightarrow \   d_h(u, v)>1,  \quad \forall \   u\in {\mathcal J}_1(x),\  v\in {\mathcal J}_1(y),
\end{equation}
or equivalently if  each $u\sim_h v$ implies  $ x\sim_h y$ whenever $x \in \mathcal J_{-1}(u)$ and $y \in \mathcal J_{-1}(v)$.
\end{definition}

\medskip
It follows that in such a graph $(X,E)$, if $x,y$ are  predecessors of $u$ (i.e., $x, y \in {\mathcal J}_{-*}(u)$) with $|x|=|y|$, then $x \sim_h y$. It is easy to see that the above condition is also equivalent to
\begin {equation}\label {eq2.2}
\max \{d_h(u, v) , 1\} \geq d_h(x, y), \quad  u\in {\mathcal J}_1(x), \ v\in {\mathcal J}_1(y).
\end{equation}
Intuitively, in an expansive rooted graph the children are drifted farther apart than their non-neighboring parents (see Figure \ref{fig:1}).

\medskip
\begin{figure}[ht]
\begin{center}
\includegraphics[scale=0.5]{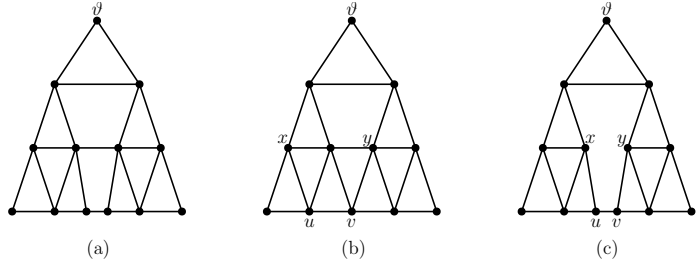}
\caption{(a) is an expansive graph, while (b) and (c) are not.} \label{fig:1}
\end{center}
\end{figure}

Clearly every rooted tree is expansive. Note that an expansive rooted graph $(X,E)$ is called a {\it pre-augmented (rooted) tree} in \cite{LW3,KLW1} if the  vertical part $(X,E_v)$ is a tree.
There are important examples in which the vertical parts of expansive graphs are not trees;  in the following  we display one of those, which arises from the well-known Bernoulli convolutions.

\begin{example} \label{ex2.2} {\rm
 Let $Y_0 = \{ \vartheta\}$,  $Y_n=\{ x = \varepsilon_1 \cdots \varepsilon_n: \varepsilon_i = 0 \hbox { or }  1\}$ for $n \geq 1$, and $Y = \bigcup_{n=0}^\infty Y_n$. Clearly there is a natural tree structure $E_v^*$ on $Y$: $x = \varepsilon_1 \cdots \varepsilon_n \in \mathcal J_1(y)$ if and only if $y=\varepsilon_1 \cdots \varepsilon_{n-1}$. For $0< \rho <1$,  let  $I_n = \{ \xi(x) = \sum_{i=1}^n \rho^n \varepsilon_i : \ x = \varepsilon_1 \cdots \varepsilon_n \in Y_n\}$.
If $0< \rho \leq 1/2$, then for each $n \geq 1$, $\xi: Y_n \to I_n$ is bijective. However for $1/2 <  \rho <1$, they may not be bijective, for example, if $\rho$ equals the golden ratio $\frac{\sqrt{5}-1}{2}$ (the positive solution of $x^2+x-1 =0$), then we have $\xi(011)=\xi (100)$.

\medskip
We define an equivalence relation $\simeq_n$ on $Y_n$ by $x\simeq_n y$ if and only if $\xi(x) = \xi(y)$. Let $X_n$ be the quotient of $Y_n$ with respect to $\simeq_n$, and use $E_v$ to denote the edge set on $X = \bigcup_{n=0}^\infty X_n$ induced by $(Y,E_v^*)$. Then $(X, E_v)$ is a vertical graph, but not a tree unless all relations $\simeq_n$ are trivial.  We can augment $(X, E_v)$ by adding a set of horizontal edges:
$$
(x, y) \in E_h \  \Leftrightarrow \ x, y \in X_n, \hbox { and  }  |\xi(x)-\xi(y) | \leq  a \rho^n \hbox{ for some } n \geq 0,
$$
where $a = \sum_{n=1}^\infty \rho^n = \frac \rho{1-\rho}$. Let $E= E_v \cup E_h$. Then $(X, E)$ is an expansive graph, because for $x, y \in X_n$ with $|\xi(x)-\xi(y)| > a \rho^n$,
\begin{flalign*}
&  & |\xi(u) - \xi(v)| > a\rho^n -\rho^{n+1} = a \rho^{n+1}, \quad  u \in {\mathcal J}_1(x), \ v \in {\mathcal J}_1(y).  \qquad \qquad \quad \ \square
\end{flalign*}
}
\end{example}
For more discussion of this example in connection with the iterations and augmented trees, the reader can refer to Appendix and  \cite {Wa}.

\bigskip

A geodesic path  $[x_i]_i$ in a rooted graph $(X,E)$ is said to be {\it convex} if $|x_i| \leq \frac 12(|x_{i-1}|+|x_{i+1}|)$ for all $i$. It is easy to see that for each convex geodesic $[x_i]_i$, there exist $u=x_k, v=x_\ell$ ($k \leq \ell$ and they can be equal) such that
\begin{equation*}
\begin {cases} x_{i-1} \in \mathcal J_1(x_i), \quad  & i \leq k;\\
 (x_{i-1}, x_i) \in E_h, \quad  & k < i \leq \ell ;\\
 x_{i-1} \in \mathcal J_{-1}(x_i), \quad & i>\ell.
\end{cases}
\end{equation*}
We denote  a convex geodesic segment from $x$ to $y$ by $\pi(x,u,v,y)$;
such geodesic may not be unique (see Figure \ref{fig:2}).

\begin{figure}[ht]
\begin{center}
\includegraphics[scale=0.5]{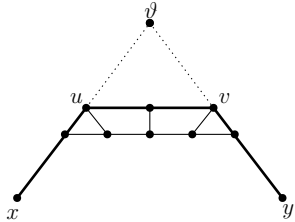}
\caption{Two convex geodesics between $x$ and $y$.} \label{fig:2}
\end{center}
\end{figure}

\medskip
The following simple result is an important property of the expansive graphs.

\begin{proposition} \label{th2.3}
Let $(X, E)$ be an expansive graph. Then any two vertices $x,y \in X$ can be joined by a convex geodesic.
\end{proposition}

\begin{proof}
Let $\pi$ be a geodesic path from $x$ to $y$. Then it is clear that $\pi$ does not contain any segment $[u,v,w]$ such that $|u|=|w|=|v| - 1$; indeed in this case, we have $d_h(u,w) \leq 1$ by the expansive property, which leads to a contradiction. Next if $\pi$ contains a segment $[u,v,w]$ such that $|v|=|w| = |u| + 1$ (or $|u| = |v| = |w| + 1$), then by the expansive property, we can replace this segment by $[u,w^-, w]$ (or $[u, u^-, w]$ respectively), where $w^- \in \mathcal J_{-1}(w)$ (see \cite[Proposition 3.4]{Ka}). Such a replacement does not change the end points or increase the length of the segment. Repeating this process, we get a convex geodesic from $x$ to $y$ eventually.
\end{proof}

\medskip

\begin{definition} \label{de2.4} \hspace{-2mm} {\rm\cite{Gr}}
On a rooted graph $(X,E)$, we define
the {\rm  Gromov product of $x, y \in X$} to be
$$
(x|y) := \frac 12 \big(|x|+ |y| - d(x,y)\big).
$$
 $(X,E)$ is said to be {\rm (Gromov) hyperbolic}  if there is $\delta \geq 0$ such that
\begin{equation*}
(x|y) \geq \min \{(x|z),(z|y)\}-\delta, \qquad \forall\ x,y,z \in X.
\end{equation*}
In this case for $a>0$ with $e^{\delta a}<\sqrt{2}$, we can define a metric $\theta_a(\cdot,\cdot)$ {\rm (Gromov metric)} on $X$ by
\begin{equation} \label{eq2.3}
\theta_a(x,y) = \inf \{{\sum}_{i=1}^n e^{-a(x_{i-1}|x_i)}: n \geq 1, \ x=x_0,x_1,\cdots,x_n = y \in X\}
\end{equation}
for distinct $x,y \in X$, and $\theta_a(x,x) = 0$ for $x \in X$.
Let $\widehat{X}$ denote the $\theta_a$-completion of $X$, and call $\partial X := \widehat{X} \setminus X$ the {\it hyperbolic boundary} of $(X,E)$.
\end{definition}

\medskip

The reader can refer to \cite{CDP, Gr, GH, Wo} for more details of the hyperbolic graphs and the hyperbolic boundaries.
We observe that if $(X, E)$ is expansive, then for $x, y \in X$ and a convex geodesic $\pi(x, u, v, y)$, we have
\begin{equation} \label{eq2.4}
(x|y) = \frac 12 \big (|x|+|y| -d(x, y) \big)= \frac 12 \big (|u| + |v| -d_h (u, v) \big ) = (u|v).
\end{equation}

Next we introduce another important notion to describe the  departing behavior of the descendant vertices in conjunction with the expansive property. Together they provide a useful criterion for the hyperbolicity of the expansive graphs.

\medskip

\begin{definition} \label{de2.5}
Let $m,k$ be two positive integers. A rooted graph $(X,E)$ is said to be  {\rm $(m,k)$-departing}  if for $x,y \in X$,
\begin{equation} \label{eq2.5}
d_h(x,y) >k \ \Rightarrow \ d_h(u,v) > 2k, \quad \forall \ u\in {\mathcal J}_m(x), \ v\in {\mathcal J}_m(y).
\end{equation}
\end{definition}

\medskip

It follows from the definitions that  every $(1,1)$-departing graph is expansive; every rooted tree is $(m,k)$-departing for any $m,k$, and every finite graph is $(m,k)$-departing for sufficiently large $m, k$. However, an infinite expansive  graph may not be $(m,k)$-departing for any $m,k$ (see Figure \ref{fig:3} for a simple example; in that graph, if both $x,y$ are on the left or right side, then  $d_h(x,y)= d_h(u, v)$ for $u \in {\mathcal J}_m(x)$ and $v \in {\mathcal J}_m(y)$, $m \geq 1$).

\medskip

\begin{figure}[ht]
\begin{center}
\includegraphics[scale=0.45]{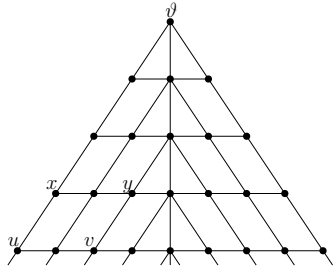}
\caption{An expansive graph that is not $(m,k)$-departing for any $m,k$.} \label{fig:3}
\end{center}
\end{figure}

\medskip
\begin{lemma} \label{th2.6}
Suppose $(X, E)$ is $(m,k)$-departing. Then it is $(m,\ell k)$-departing for any positive integer $\ell$.
\end{lemma}

\begin{proof}
 Let $x,y \in X$, $u \in \mathcal J_m(x)$ and $v \in \mathcal J_m(y)$ that satisfy $d_h(u, v) \leq 2 \ell k$. Then there exists a set of vertices $\{z_0, z_1, \cdots, z_s\} \subset X_{|x|+m}$ such that $z_0=u$, $z_s=v$, $s \leq \ell$ and $d_h(z_{i-1}, z_i) \leq 2k$ for $i=1,2,\cdots, s$.
By the $(m,k)$-departing property, we have
\[
d_h(x,y) \leq \sum_{i=1}^s d_h(z_{i-1}^{(-m)}, z_i^{(-m)}) \leq s k \leq \ell k,
\]
where $z_i^{(-m)} \in \mathcal J_{-m}(z_i)$ and $z_0^{(-m)} = x$, $z_s^{(-m)} = y$. This completes the proof.
\end{proof}

\medskip
It follows that $(m,1)$-departing implies $(m,k)$-departing for any $k$. Also, it is straightforward to check inductively that $(1,k)$-departing implies $(m, k)$-departing for any $m$. Hence we have

\medskip

\begin {corollary} \label {th2.7}
Suppose $(X, E)$ is $(1,1)$-departing. Then it is $(m,k)$-departing for any positive integers $m,k$.
\end{corollary}

\medskip

We give two examples to illustrate the $(m,k)$-departing property.

\begin{example} \label{ex2.8} {\bf (SG graph)} {\rm
Let $X =\bigcup_{n=0}^\infty \Sigma^n$  with $\Sigma = \{0,1,\cdots,d\}$ ($d \geq 1$) be the symbolic space representing the $d$-dimensional Sierpinski gasket $K$, and let $K_x$ be the cell associated to the word $x \in X$. Let $E_v$ be the natural tree structure on $X$, and define $E_h = \{(x, y)\in X\times X: |x| = |y|, \ x \neq y, \hbox{ and } K_x \cap K_y \not = \emptyset\}$ \cite {Ka}. Consider the rooted graph $(X, E)$ with $E= E_v \cup E_h$. It is easy to show  that  $(X, E)$ is expansive and $(1,1)$-departing.

\vspace {0.1cm}

Indeed, it is easy to check that  $u \sim_h v$ if and only if either (i) $u = wi, v =wj$ for some $w \in X, i, j \in  \Sigma$, or (ii) $u =wij^k, v= wji^k$ for some $w\in X, i, j  \in \Sigma, i\not =j, k \geq 1$ \cite{DS1, Ki1}.
Note that (i) implies $u^- = v^-$, and (ii) implies  $d_h (u^-, v^-) =1$. Therefore $u \sim_h v$ implies $u^- \sim_h v^-$, i.e., $(X,E)$ is expansive. Moreover, for any horizontal segment $[x,z,y]$, we can conclude that one of $x^-, y^-$ must equal $z^-$, hence $d_h(x^-, y^- ) \leq 1$. This shows that $(X,E)$ is $(1,1)$-departing. \hfill $\square$}
\end{example}

\begin{example} \label{ex2.9} {\bf (Hata tree)} {\rm The Hata tree $K$ is defined on ${\mathbb C}$ by the iterated function system $S_1(z) = c\bar z$ and $S_2(z) = (1-|c|^2) \bar z + |c|^2$ where $|c|, |1-c| \in (0,1)$ (see the left figure of Figure \ref{fig:ht}, the detailed description of $K$ can be found in \cite [p.16]{Ki1}). The graph $(X, E)$ is the symbolic space $X$ with $E= E_v\cup E_h$, where $E_h = \{(x, y)\in X\times X: |x| = |y|, \ x \neq y, \hbox{ and }  K_x \cap K_y \not = \emptyset\}$ (the right picture). It is clear that $(X, E)$ is expansive.

\vspace{0.1cm}

From the graph, we see that  $d_h(11,  22)= d_h (112,  221)=2$ (the 212-branch sticks out and is thus  not counted). Hence $(X, E)$ is not $(1, 1)$-departing.
However, it is not hard to see that $(X,E)$ is $(m, 1)$-departing for $m \geq 2$ (check level 2, then use the similarity of the graph to conclude for all levels), since the further separation of descendants overcomes the missing count of those outlying branches. \hfill $\square$}
\end{example}

\begin{figure}[htbp]
\begin{center}
    \begin{minipage}[t]{0.60\linewidth}
\includegraphics[height=3.0cm]{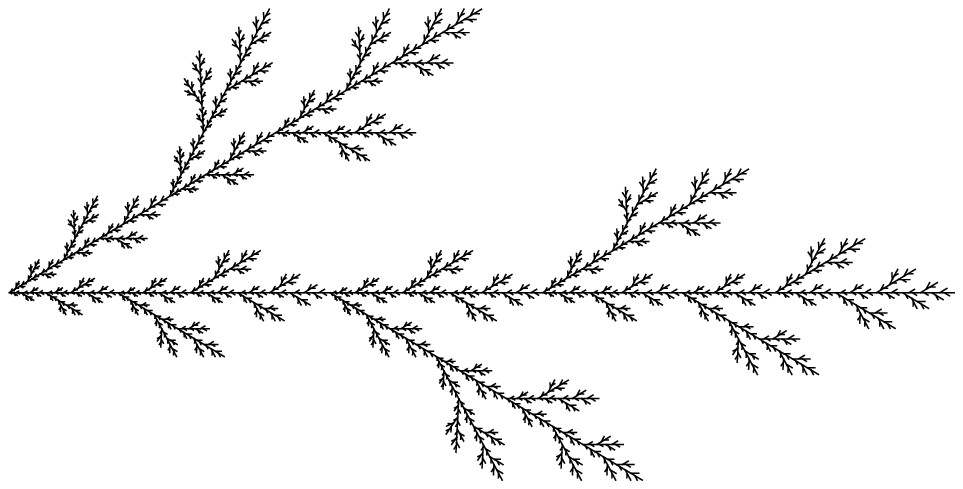}
    \end{minipage}
 \hspace{-3.5cm}
    \begin{minipage}[t]{0.30\linewidth}
           {\includegraphics[height=3.8cm]{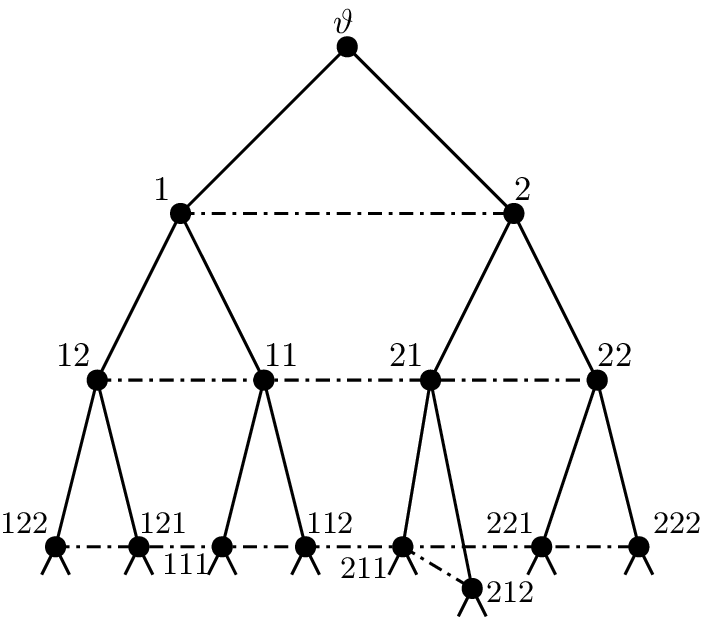}}
    \end{minipage}
\caption{Hata tree ($c=0.5+0.3i$) and the graph $(X,E)$.} \label{fig:ht}
\end{center}
\end{figure}

\begin{proposition} \label{th2.10}
Suppose $(X,E)$ is $(m,k)$-departing,  then the lengths of all horizontal geodesics in $(X,E)$ are bounded by $L=L(m,k):=\lceil\frac{2m+1}{k}\rceil k+2m$. In particular, for $(m,1)$-departing rooted graphs,  $L = 4m+1$.
\end{proposition}

\begin{proof}
Suppose otherwise, then there exists a horizontal geodesic $\pi(x,y)$ with length $L+1$. Note that $\pi(x,\vartheta) \cup \pi(\vartheta, y)$ is a path joining $x$ and $y$. Comparing the lengths of two paths, we have $2 |x| \geq L+1 > 2m$, i.e., $|x| > m$. Let $x^{(-m)} \in \mathcal J_{-m}(x)$ and $y^{(-m)} \in \mathcal J_{-m}(y)$. Then there exists a horizontal path joining $x^{(-m)}$ and $y^{(-m)}$ (otherwise $d_h(x^{(-m)}, y^{(-m)}) = + \infty$, and Lemma \ref{th2.6} implies that $d_h(x,y) = + \infty$, a contradiction). Now consider a new path: $x \to x^{(-m)}$ (along horizontal edges) $\to y^{(-m)} \to y$, and by comparing the length of this new path with the one of geodesic $\pi(x,y)$ we have
$$
2 m + d_h(x^{(-m)}, y^{(-m)}) \geq L+1.
$$
It follows that $d_h(x^{(-m)}, y^{(-m)}) \geq \lceil\frac{2m+1}{k}\rceil k+1 > \lceil\frac{2m+1}{k}\rceil k$. Making use of Lemma \ref{th2.6}, we have $d_h(x,y) > 2\lceil\frac{2m+1}{k}\rceil k \geq L+1$. This is a contradiction, and completes the proof.
\end{proof}

\medskip

The following theorem provides two useful criteria for the hyperbolicity of expansive graphs.

\medskip

\begin{theorem} \label{th2.11}
Let $(X,E)$ be an expansive graph. Then the following assertions are equivalent.

\ \ (i) $(X,E)$ is hyperbolic;

\ (ii) $\exists$ $L<\infty$ such that the lengths of all horizontal geodesics are
 bounded by $L$;

(iii) $(X,E)$ is $(m,k)$-departing for some positive integers $m,k$.
\end{theorem}

\medskip

\begin{proof}
(i) $\Leftrightarrow$ (ii) follows from a similar proof as in \cite[Theorem 2.3]{LW1}.

\vspace{1mm}

(iii) $\Rightarrow$ (ii) follows from Proposition \ref{th2.10}.

\vspace{0.1cm}

(ii) $\Rightarrow$ (iii): We claim that $(X,E)$ is $(L+1,L+2)$-departing. Indeed, let $x,y \in X$, $x' \in \mathcal J_{L+1}(x)$ and $y' \in \mathcal J_{L+1}(y)$ satisfying $L+2< d_h(x',y') \leq 2(L+2)$ (see Figure \ref{fig:pf}). By Proposition \ref{th2.3}, there exists a convex geodesic segment $\pi(x',u,v,y')$ between $x'$ and $y'$, and  $u \neq x'$ (by the first inequality and (ii)). Let $u, v \in X_j$. Then
$$
2(L+2) \geq d_h(x',y') > d(x',y') = 2(|x'|- |j|)+d_h(u,v)\geq 2(|x'|-j).
$$
As $\ell:=|x'|-j \leq L+1 = |x'|-|x|$, we have $j \geq |x|$. Let $u' \in \mathcal J_*(x) \cap \mathcal J_{-*}(x') \cap  X_j$ and $v' \in \mathcal J_*(y) \cap \mathcal J_{-*}(y') \cap X_j$.
Since $x' \in \mathcal J_\ell (u) \cap \mathcal J_\ell (u')$,  $u$ and $u'$ are predecessors of $x'$, and we have $ u \sim_h u'$ by the expansive property. Similarly $ v\sim_h v'$. Hence by \eqref{eq2.2} and (ii),
$$
d_h(x,y) \leq \max\{d_h(u',v'),1\}  \leq d_h(u,v) + 2 \leq L+2.
$$
This proves the claim.
\end{proof}

\begin{figure}[ht]
\begin{center}
\includegraphics[scale=0.6]{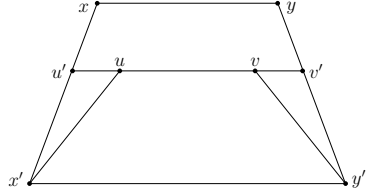}
\caption{Illustration for the proof of Theorem \ref{th2.11}.} \label{fig:pf}
\end{center}
\end{figure}

We will call the graphs in the above theorem {\bf expansive hyperbolic graphs}.

\bigskip

\section{Hyperbolic boundaries}
\label{sec:3}

In an infinite graph $(X,E)$ with root $\vartheta$, we let
\begin{equation} \label{eq3.1}
\mathcal R_v := \{{\bf x} = [x_i]_{i=0}^\infty: x_0 = \vartheta, \hbox{ and } x_{i+1} \in \mathcal J_1(x_i),\ \forall\ i \geq 0\}
\end{equation}
denote the class of (geodesic) rays starting from the root $\vartheta$. For brevity, we shall use the bold symbols such as ${\bf x},{\bf y},{\bf z}$ to denote the rays $[x_i]_i,[y_i]_i,[z_i]_i$ in $\mathcal R_v$ respectively.

\bigskip

 In this section, we assume that $(X,E)$ is hyperbolic. It follows from \eqref{eq2.3} that
\begin{equation} \label{eq3.2}
\theta_a(x,y) \asymp e^{-a(x|y)}, \qquad \forall \ x,y \in X
\end{equation}
(see \cite[p.245]{Wo}), and hence the topology on the $\theta_a$-completion $\widehat{X}$ is independent of the value of $a$ (as $\theta_a(\cdot,\cdot) \asymp \theta_b(\cdot,\cdot)^{a/b}$ for $a,b>0$); it is also known that both $\widehat{X}$ and $\partial X$ are compact.

\medskip

Every ray in $\mathcal R_v$ is $\theta_a$-Cauchy.
Note that for ${\bf x}, {\bf y} \in \mathcal R_v$, the Gromov product $(x_i|y_i)$ is increasing in $i$ by the triangle inequality. By using this we define $({\bf x}|{\bf y}) := {\lim}_{i \to \infty} (x_i|y_i)$.
We say that ${\bf x},{\bf y} \in \mathcal R_v$ are {\it equivalent} if they converge to the same point in $\partial X$; this holds if and only if $({\bf x}|{\bf y}) = \infty$. With such equivalence, the quotient of $\mathcal R_v$ is identified with $\partial X$.

\medskip

 For an integer $k \geq 0$ and two rays ${\bf x}, {\bf y} \in \mathcal R_v$, we define
\begin{equation} \label{eq3.3}
|{\bf x} \vee {\bf y}|_k := \sup\{i \geq 0: d_h(x_i,y_i) \leq k\}.
\end{equation}
Clearly if $E=E_v$, $|{\bf x} \vee {\bf y}|_k = \sup\{i \geq 0: x_i = y_i\}$ for all $k \geq 0$, and equals $({\bf x}|{\bf y})$ when $(X,E)$ is a tree. Note that if $|{\bf x} \vee {\bf y}|_k = \infty$, i.e., $d_h(x_i,y_i) \leq k$ for all $i \geq 0$, then
by
\begin{equation} \label{eq3.4}
(x_i|y_i) = |x_i| - \frac 12 d(x_i, y_i) \geq i - \frac k2,
\end{equation}
we have $\lim_{i \to \infty} (x_i|y_i) =\infty$, i.e., ${\bf x}$ and ${\bf y}$ are equivalent.

\medskip

Furthermore, if the hyperbolic graph $(X,E)$ is expansive, then  by Theorem \ref{th2.11}, it is $(m,k)$-departing for some $m,k > 0$. This provides a more concrete characterization of the equivalence classes in $\mathcal R_v$ as follows, which is used for some estimations in the sequel.

\begin{proposition} \label{th3.1}
Suppose the rooted graph $(X,E)$ is expansive and $(m,k)$-departing. Then there exists a constant $D_0=D_0(m,k) > 0$ such that
\begin{equation}  \label{eq3.5}
| ({\bf x}|{\bf y}) - |{\bf x} \vee {\bf y}|_k | \leq D_0, \qquad \forall\ {\bf x}, {\bf y} \in \mathcal R_v.
\end{equation}
Consequently, two rays ${\bf x},{\bf y}$ in $\mathcal R_v$ are equivalent if and only if $d_h(x_i,y_i) \leq k$ for all $i \geq 0$.
\end{proposition}

\begin{proof}  Note that $|{\bf x} \vee {\bf y}|_k = \infty$ implies $({\bf x}|{\bf y}) = \infty$. Hence we need only prove \eqref{eq3.5} for the case that $|{\bf x} \vee {\bf y}|_k<\infty$. Set $\ell := |{\bf x} \vee {\bf y}|_k+1$. Then $d_h(x_{\ell-1}, y_{\ell-1}) \leq k$, and $d_h(x_\ell, y_\ell) > k$. Using \eqref{eq3.4} for $i = \ell-1$, we have
$$
({\bf x}|{\bf y}) \geq (x_{\ell-1} | y_{\ell-1}) \geq \ell-1-k/2 =|{\bf x} \vee {\bf y}|_k - k/2.
$$
On the other hand, by using the $(m,k)$-departing in Lemma \ref{th2.6} repeatedly, we have
\begin{equation} \label{eq3.6}
d_h(x_{mj+\ell}, y_{mj + \ell}) > 2^j k, \quad \forall \  j \geq 1.
\end{equation}
 For $j \geq 0$, we choose $u,v \in X$ such that $\pi(x_{mj+\ell},u,v,y_{mj+\ell})$ is a convex geodesic.
Then $x_{|u|} \sim_h u$ and $v \sim_h y_{|u|}$ (by the remark after Definition \ref {de2.1}).  By Proposition \ref{th2.10}, there exists $L=L(m,k)>0$ such that
\begin{equation}\label{eq3.7}
d_h(x_{|u|}, y_{|v|}) \leq d_h(x_{|u|}, u) + d_h(u, v) + d_h(v, y_{|v|}) \leq 1 + L +1 = L+2.
\end{equation}
Let $j_0 = j_0(m,k)$ be the smallest integer such that $2^{j_0}k \geq L+2$. It follows from \eqref {eq3.6} and \eqref{eq3.7} that $|u|= |v| \leq mj_0 +\ell-1$.  Hence by
\eqref {eq2.4}, we have
\begin{equation*} 
(x_{mj+\ell}| y_{mj+\ell}) = (u|v) =|u| - \frac{1}{2} d_h(u, v) \leq m j_0 + \ell -1,
\end{equation*}
 and this implies that $({\bf x}|{\bf y}) \leq m j_0 + \ell -1 = |{\bf x} \vee {\bf y}|_k+ mj_0$. Setting $D_0 = \max\{k/2, mj_0\}$, inequality \eqref{eq3.5} follows.
\end{proof}

\medskip

By Proposition \ref{th3.1}, we can extend the Gromov product to $\widehat{X}$ by letting
\begin{equation} \label{eq3.8}
(x|\xi) = \sup\{\lim_{i \to \infty} (x|x_i)\}, \qquad (\xi|\eta) = \sup\{\lim_{i \to \infty} (x_i|y_i)\},
\end{equation}
where $x \in X$, $\xi,\eta \in \partial X$, and the supremum is taking over all rays  $[x_i]_i$ and $[y_i]_i$ in ${\mathcal R}_v$ that converge to $\xi$ and $\eta$ respectively;
the differences of the limits for different rays are at most $k$ if $(X,E)$ is $(m,k)$-departing.
Then the Gromov metric $\theta_a(\cdot,\cdot)$ on $\widehat{X}$ satisfies the same estimate as in \eqref{eq3.2}.

\bigskip

Throughout the rest of the paper, we will assume that $\mathcal J_1(x) \neq \emptyset$ for all $x \in X$.  In this section we will investigate  the metric doubling property of $(\partial X,\theta_a)$. For $x \in X$ and $\xi \in \partial X$, we extend the partial order $\prec$ in Section \ref{sec:2} by writing $\xi \prec x$ if $x$ lies on some ray $[x_i]_i \in \mathcal R_v$ that converges to $\xi$. Define
 $$
 \mathcal J_\partial(x):=\{\xi \in \partial X: \xi \prec x\}
 $$
and call it the {\it cell} of $x$ in $\partial X$.  Then it is clear that each $\mathcal J_\partial(x)$ is a nonempty compact subset of $\partial X$.
We remark that in the context of IFS with attractor $K$  and an associated graph structure $(X, E)$ of the symbolic space, $\mathcal J_\partial(x)$ plays the role of the $x$-cell $K_x$  in $K$ (see Appendix);  this will be discussed in detail in Section \ref{sec:4}.

\medskip

\begin{proposition} \label{th3.2}
Let $(X,E)$ be an $(m,k)$-departing expansive graph. Then there exists a constant $\gamma >0$ (depending on $a$) such that for $x,y \in X_n$, $n\geq 1$,
\begin{equation*}
 d_h(x,y) > k \quad \Rightarrow \quad {\rm dist}_{\theta_a}(\mathcal J_\partial(x), \mathcal J_\partial(y)) > \gamma e^{-an}.
\end{equation*}
\end{proposition}

\begin{proof}
 For $\xi \in \mathcal J_\partial(x)$ and $\eta \in \mathcal J_\partial(y)$, we choose two rays ${\bf x}, {\bf y} \in \mathcal R_v$ that pass through $x,y$ and converge to $\xi,\eta$ respectively. Then $|{\bf x} \vee {\bf y}|_k < n$. It follows from \eqref{eq3.5} that
$$
(\xi|\eta) \leq ({\bf x}|{\bf y})+k \leq |{\bf x} \vee {\bf y}|_k+k+D_0 < n+k+D_0.
$$
Hence $\theta_a(\xi,\eta) \geq c e^{-a(\xi|\eta)} > \gamma e^{-a|x|}$ with $\gamma = ce^{-a(k+D_0)}$.
\end{proof}

\medskip

For an integer $k \geq 0$, define the {\it $k$-shadow} of $x$ in $\partial X$ by
\begin{equation} \label{eq3.9}
\mathcal J^k_\partial(x) := {\bigcup}\{\mathcal J_\partial(y):d_h(x,y) \leq k\}.
\end{equation}
(Hence $\mathcal J^0_\partial(x) = \mathcal J_\partial(x)$.)

\medskip

\begin{proposition} \label{th3.3}
Let $(X,E)$ be an $(m,k)$-departing  expansive graph. Then
there exists a constant $C \geq 1$ (depending on $a$) such that
\begin{equation} \label{eq3.10}
B_{\theta_a}(\xi,C^{-1}e^{-a|x|}) \subset \mathcal J^k_\partial(x) \subset B_{\theta_a}(\xi,Ce^{-a|x|}), \quad \forall\ x \in X, \ \xi \in \mathcal J_\partial(x).
\end{equation}
\end{proposition}

\begin{proof}
Suppose $x \in X_n$.
For $\eta \in \partial X \setminus \mathcal J_\partial^k(x)$, we choose $y \in X_n$ with $\eta \in \mathcal J_\partial(y)$. Then $d_h(x,y) > k$, and by Proposition \ref{th3.2} we have
$$
\theta_a(\xi,\eta) \geq {\rm dist}_{\theta_a}(\mathcal J_\partial(x), \mathcal J_\partial(y)) > \gamma e^{-an}, \qquad \forall \ \xi \in \mathcal J_\partial(x).
$$
This shows that $B_{\theta_a}(\xi,\gamma e^{-a|x|}) \subset \mathcal J_\partial^k(x)$ for all $\xi \in \mathcal J_\partial(x)$.

\vspace{1mm}

For $\zeta \in \mathcal J^k_\partial(x)$, we choose $z \in X_n$ that satisfies $d_h(x,z) \leq k$ and $\zeta \in \mathcal J_\partial(z)$. Then it follows from \eqref{eq3.8} that
$$
(\xi|\zeta) \geq (x|z) = n-\frac 12 d(x,z) \geq n-\frac k2, \qquad \forall\ \xi \in \mathcal J_\partial(x).
$$
Hence $\theta_a(\xi, \zeta) \leq C_1e^{-a(\xi|\zeta)} \leq C_2e^{-a|x|}$ with $C_2=C_1 e^{ak/2}$, and $\mathcal J^k_\partial(x) \subset B_{\theta_a}(\xi,C_2 e^{-a|x|})$ for all $\xi \in \mathcal J_\partial (x)$.
\end{proof}

\medskip

We remark that in \eqref{eq3.10}, $\mathcal J^k_\partial(x)$ cannot be simply replaced by $\mathcal J_\partial(x)$, as is shown in the following simple example that the left inclusion does not hold.

\medskip

\begin{example} \label{ex3.4}
{\rm Let $(X,E)$ be the rooted graph as in Figure \ref{fig:4}. It is an expansive graph, in which any (nontrivial) horizontal geodesic has length $1$. By Theorem \ref{th2.11}, the graph is hyperbolic.  Let $x = x_1$ be as in the figure. Then it is seen that $\mathcal J_\partial(x) = \{\xi\}$, where $\xi$ is the limit of the ray ${\bf x} = [x_i]_{i=0}^\infty$ in $\partial X$. There is another ray ${\bf a} =[a_i]_{i=0}^\infty$ converging to $\xi$.

\vspace {0.1cm}

For $n \geq 1$, let $\eta_n$ be the limit of the ray $[\vartheta, a_1, a_2,\cdots,a_{n+1},b_n,\cdots]$ in $\partial X$. Then $\partial X = \{\xi,\eta_1,\eta_2,\cdots\}$.
By \eqref{eq3.8}, we have $(\xi | \eta_n) = n + 1$ and $\theta_a(\xi, \eta_n) \asymp e^{-a(n + 1)} \to 0$ as $n \to \infty$. We see that  $\mathcal J_\partial (x)$ cannot contain any  ball $B(\xi, r), r>0$. \hfill $\square$}

\begin{figure}[ht]
\begin{center}
\includegraphics[scale=0.35]{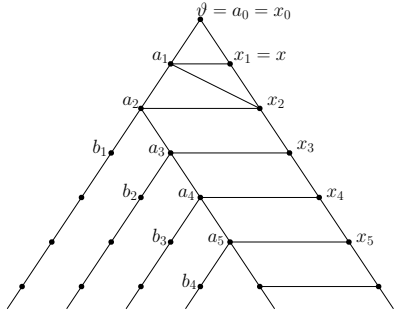}
\caption{$\mathcal J_\partial(x)$ does not contain any ball.} \label{fig:4}
\end{center}
\end{figure}
\end{example}

\begin{definition} \label{de3.5} \hspace{-2mm} {\rm \cite{He}}
A metric space $(M,\rho)$ is called {\rm (metric) doubling} if there is an integer $\ell>0$ such that for any $\xi \in M$ and $r>0$, the ball $B_\rho(\xi,r)$ can be covered by a union of at most $\ell$ balls of radii $r/2$.
\end{definition}

\medskip

\begin{theorem} \label{th3.6}
Suppose $(X,E)$ is an expansive hyperbolic graph, and has bounded degree (i.e., $\sup_{x\in X}\deg(x) < \infty$).  Then the hyperbolic boundary $(\partial X, \theta_a)$ is doubling.
\end{theorem}


\begin{proof}
Fix $a>0$  with $e^{\delta a} < \sqrt{2}$, and suppose $(X,E)$ is $(m,k)$-departing (by Theorem \ref{th2.11}). Let $C \geq 1 $ be the constant as in \eqref {eq3.10}.  For $r \in (0,1]$, let $m^*(r)$\ ($m_*(r)$) be the smallest (largest respectively) nonnegative integers such that
$$
C^{-1} e^{-a(m^*(r)+1)} < r < 2Ce^{-a(m_*(r)-1)}.
$$
It follows that $m_*(r)-m^*(r) < \lceil\frac{\log(2C^2)}{a}\rceil+2 =:\ell_1$. For a ball $B_{\theta_a}(\xi,r) \subset \partial X$, let $x$ be the vertex such that $|x| = m^*(r)$ and $\xi \in \mathcal J_\partial(x)$.  We claim that $B_{\theta_a}(\xi,r) \subset \mathcal J_\partial^k(x)$. Indeed, if $m^*(r) = 0$, then $x = \vartheta$, and trivially $B_{\theta_a}(\xi,r) \subset \partial X = \mathcal J_\partial^k(\vartheta)$; if $m^*(r) \geq 1$, then by Proposition \ref{th3.3} we have $B_{\theta_a}(\xi,r) \subset B_{\theta_a}(\xi,C^{-1}e^{-am^*(r)}) \subset \mathcal J_\partial^k(x)$ as well. By the claim and Proposition \ref{th3.3}, we have
\begin{align*}
B_{\theta_a}(\xi,r) &\subset \mathcal J^k_\partial(x) = {\bigcup}_{y:d_h(x,y) \leq k}  {\bigcup}_{ z \in \mathcal J_{\ell_1}(y)}\mathcal J_\partial(z)\\
&\subset {\bigcup}_{y:d_h(x,y) \leq k} \ {\bigcup}_{ z \in \mathcal J_{\ell_1}(y)} B_{\theta_a}(\eta_z,Ce^{-am_*(r)}) \qquad  (\hbox  {by \eqref{eq3.10}})\\
&\subset {\bigcup}_{y:d_h(x,y) \leq k} \ {\bigcup}_{ z \in \mathcal J_{\ell_1}(y)}B_{\theta_a}(\eta_z,r/2),
\end{align*}
where each $\eta_z$ is chosen arbitrarily from $\mathcal J_\partial(z)$. Let $t = \sup_{x \in X} \deg(x)$, which is finite by assumption. Then
$$
\# \{y: d_h(x,y) \leq k\} \leq t^k \quad \hbox{and} \quad \# \mathcal J_{\ell_1}(y) \leq t^{\ell_1}.
$$
Hence the ball $B_{\varrho_a}(\xi,r)$ is covered by a union of at most $t^{k+\ell_1}$ many balls of radii $r/2$, and $(\partial X, \theta_a)$ is doubling.
\end{proof}

\noindent {\bf Remark 1.}
A similar result was proved in another setting  by Bonk and Schramm \cite[Theorem 9.2]{BS}; here our proof on expansive hyperbolic graphs is more straightforward.

\medskip

\noindent {\bf Remark 2.}
Note that the doubling property of $(\partial X, \theta_a)$ does not imply the bounded degree property of $(X,E)$. For example, if we take $E_h=\{(x,y) \in X \times X \setminus \Delta: |x|=|y|\}$, then it is trivially an expansive hyperbolic graph, and $\partial X$ is a singleton (so it is doubling trivially); but $\deg(x) \geq \#X_n -1$ for $x \in X_n$, which has no bound if $\#X_n$ tends to $\infty$.

However, this converse can be verified  with some further separation properties among the cells (Theorem \ref{th5.5}).

\bigskip

\section{Index maps and augmented graphs}
\label{sec:4}

In this section, we  fix a complete metric space $(M, \rho)$, and let ${\mathcal C}_M$ denote the family of all nonempty compact subsets of $M$. By our convention in last section, any graph $(X,E)$ mentioned below is assumed to satisfy ${\mathcal J}_\partial (x) \not = \emptyset$ for all $x\in X$.

\begin{definition} \label{de4.1}
Let $(X, E_v)$ be a vertical graph. A map $\Phi : X \rightarrow {\mathcal C}_M $ is called an {\rm index map} (on $(X,E_v)$ over $(M,\rho)$) if it satisfies

\ (i) \ $\Phi(y) \subset \Phi(x)$ for all $x \in X$ and $y \in \mathcal J_1(x)$;

(ii) \ $\bigcap_{i=0}^\infty \Phi(x_i)$ is a singleton for all ${\bf x} = [x_i]_i \in \mathcal R_v$.

\vspace {0.1cm}

\noindent  In particular, such $\Phi$ is called {\rm saturated} if (i) is strengthened to $\Phi(x) = {\bigcup}_{y \in \mathcal J_1(x)} \Phi(y)$.

\vspace {0.1cm}
\noindent We call
$K := {\bigcap}_{n=0}^\infty\, \Big({\bigcup}_{x \in X_n}\, \Phi(x)\Big)$
the {\rm attractor} of $\Phi$, and $K_x:= \Phi(x) \cap K$ a {\rm cell} of $K$.

\end{definition}

\noindent {\bf Remark 1.}
An index map  $\Phi$ induces another index map $\Phi' : X \rightarrow {\mathcal C}_K$ with $\Phi'(x) = K_x$; the image $\{K_x\}_{x\in X}$ is the family of cells of $K$ indexed by $X$.  Since $\Phi$ and $\Phi'$ {\it behave the same at infinity}  (i.e., $\bigcap_{i=0}^\infty \Phi(x_i) = \bigcap_{i=0}^\infty \Phi'(x_i)$ for all ${\bf x} = [x_i]_i \in \mathcal R_v$), we use these two interchangeably.

\medskip

\noindent {\bf Remark 2.}
 For a saturated index map $\Phi$, the attractor $K = \Phi(\vartheta)$, and the cell $K_x = \Phi(x)$ for all $x \in X$.
In fact, every index map $\Phi$ induces a saturated index map $\widetilde\Phi: X \to \mathcal C_K$ with
\begin{equation} \label{eq4.1}
\widetilde\Phi(x) := {\bigcap}_{n=0}^\infty\, \Big({\bigcup}_{y \in {\mathcal J}_n(x)}\, \Phi(y)\Big), \qquad x \in X,
\end{equation}
which also behaves the same as $\Phi$ at infinity.
It is clear that $\widetilde\Phi(x) \subset K_x$ for all $x\in X$, but the reverse inclusion does not hold in general.

\medskip

The index map $\Phi$ defines a mapping $\kappa_0: \mathcal R_v \to K$ by
\begin{equation} \label{eq4.2}
\{\kappa_0({\bf x})\} = {\bigcap}_{i=0}^\infty \, \Phi(x_i), \qquad \forall\ {\bf x} = [x_i]_i \in \mathcal R_v.
\end{equation}
Using the local finiteness of $(X,E_v)$ and a diagonal argument (cf.~\cite{Wo, LW3}), we can show that the image of $\kappa_0$ is equal to $K$.
From Section \ref{sec:3}, we see that for an expansive hyperbolic graph $(X,E)$,  the hyperbolic boundary $\partial X$ can be identified with a quotient set of $\mathcal R_v$.  Hence the induced map $ \kappa: \partial X \to K$ from the quotient is well-defined if $ \kappa_0$ satisfies: ${\bf x}$, ${\bf y} \ \hbox {are equivalent}\ \Rightarrow \ \kappa_0 ({\bf x}) = \kappa_0 ({\bf y})$; furthermore $\kappa: \partial X \to K$  is one-to-one if  the converse is also satisfied.
With these,  we see that if $\kappa_0$ satisfies
\begin{equation} \label{eq4.3}
{\bf x} \ \hbox {and} \  {\bf y}  \ \hbox{are equivalent} \quad \Leftrightarrow \quad \kappa_0 (\bf x) = \kappa_0 (\bf y),
\end{equation}
then it induces a bijection $\kappa: \partial X \to K$ via the quotient.

\medskip

\begin{definition} \label{de4.2}
We call $(X,E,\Phi)$ an {\rm admissible index triple} if it satisfies

\ \ (i) \ $(X,E)$ is an expansive hyperbolic graph;

\ (ii) \ $\Phi: X \to {\mathcal C}_M$ is an index map on $(X,E_v)$ over $(M,\rho)$;

(iii) \  $\kappa : \partial X \to K$  is a  well-defined bijection,  i.e., \eqref{eq4.3} holds.

\noindent In such case, $(X,E)$ is said to be an {\rm admissible (augmented) graph}  (associated to $\Phi$); if $(X, E_v)$ is a tree, then we call $(X,E)$ an {\rm admissible augmented tree}.
\end{definition}

By Remark 2,
\begin{equation} \label{eq4.4}
\kappa (\mathcal J_\partial (x)) = \widetilde \Phi(x) \subset \Phi(x), \qquad \forall \ x \in X; 
\end{equation}
and the inclusion is an ``=" if and only if the index map $\Phi$ is saturated.  Via the bijection $\kappa$, the Gromov metric $\theta_a$ on $\partial X$  defines naturally a metric $\tilde\theta_a$ on the attractor $K$ by
\begin{equation} \label{eq4.9}
\tilde\theta_a(\xi,\eta) = \theta_a(\kappa^{-1}(\xi),\kappa^{-1}(\eta)), \qquad \xi,\eta \in K.
\end{equation}

\begin{proposition} \label{th4.3}
For an admissible index triple, the bijection $\kappa$ is a homeomorphism.
\end{proposition}

 We  will prove this in \cite{KLWa}, as the proof requires more preparatory work and we will not need the proposition here.

\bigskip

For a subset $E$ in $(M, \rho)$, we denote the diameter of $E$ by $|E|_\rho$ (or simply by $|E|$).  In Definition \ref{de4.1}, we see that the family  $\{\Phi(x)\}_{x\in X}$ satisfies ${\lim}_{n \to \infty} {\sup}_{x \in \mathcal J_n} |\Phi(x)|_\rho = 0$. For $b \in (0,\infty)$, we say that $\{\Phi(x)\}_{x\in X}$ (or $\Phi$) is of {\it exponential type-$(b)$} (under $\rho$) if the diameter $|\Phi(x)|_\rho$ is decreasing in a rate of  $e^{-b|x|}$,  i.e., $|\Phi(x)|_\rho = O(e^{-b|x|})$ as $|x| \to \infty$, and call $\Phi$ an {\it exponential type} if it is of type-$(b)$ for some $b \in (0,\infty)$.

\bigskip

 We mainly consider the following two classes of expansive graphs associated to index maps, which are motivated by the {\it augmented trees} of the IFS's \cite {Ka,LW1, LW3, Wa} (see Appendix).

\begin{definition} \label{de4.4}
 Let $\Phi$  be an index map on the vertical graph $(X, E_v)$. We define a horizontal edge set by
\begin{equation} \label{eq4.5}
E_h^{(\infty)}  :={\bigcup}_{n=1}^\infty \big\{(x,y) \in X_n \times X_n \setminus \Delta: \Phi(x)\cap \Phi(y)\not = \emptyset \big\},
\end{equation}
and let $E^\infty = E_v \cup E_h^{(\infty)}$. We call $(X, E^{(\infty)})$ an {\rm $AI_{\infty}$-graph}, {\rm augmented index graph of type-$(\infty)$} (or {\rm intersection type}).

\vspace{0.1cm}
In addition, assume that $\Phi$ is of exponential type-$(b)$. For a fixed $\gamma >0$, we define
\begin{equation} \label{eq4.6}
E_h^{(b)}  :={\bigcup}_{n=1}^\infty \big\{(x,y) \in X_n \times X_n \setminus \Delta: {\rm dist}_\rho (\Phi(x),\Phi(y)) \leq \gamma e^{-bn}\big\}.
\end{equation}
Let $E^{(b)} = E_v \cup E^{(b)}_h$. We call  $(X, E^{(b)})$ an {\rm $AI_b$-graph},   {\rm  augmented index graph of type-$(b)$}.
\end{definition}

\bigskip

It is clear that both $(X, E^{(\infty)})$ and $(X, E^{(b)})$ are expansive.  Comparing the two definitions, the $AI_\infty$-graph is more intuitive but needs more  information on the neighborhood of $\Phi(x), \ x \in X$ under the given metric $\rho$; the $AI_b$-graph is more flexible on the neighboring cells, which actually makes it easier to handle.  First we prove

\begin{theorem} \label{th4.5}
For an index map $\Phi$ on $(X,E_v)$ over $(M, \rho)$ of exponential type-$(b)$, the associated  $AI_b$-graph is  $(m,1)$-departing, and is an admissible graph.
Moreover,  $\kappa: (\partial X, \theta_a) \to (K, \rho)$  is a  H\"older equivalence with exponent $b/a$, i.e.,
\begin{equation} \label{eq4.7}
\rho(\kappa(\xi),\kappa(\eta))^{a/b} \asymp  \theta_a(\xi,\eta), \qquad \forall\ \xi,\eta \in \partial X.
\end{equation}
\end{theorem}

\begin{proof} It is easy to check from the definition of $E^{(b)}_h$ that $(X, E^{(b)})$ is expansive. To show that it is $(m,1)$-departing for some $m\geq 1$, let $\delta_0:={\sup}_{z \in X}\, e^{b|z|} |\Phi(z)|$.
Let $u \in \mathcal J_m(x)$ and $v \in \mathcal J_m(y)$ with $d_h(u,v)=2$. Using the triangle inequality twice, we have
$$
{\rm dist}(\Phi(x),\Phi(y)) \leq {\rm dist}(\Phi(u),\Phi(v)) \leq (2\gamma+\delta_0)e^{-b(|x|+m)} \leq \gamma  e^{-b|x|},
$$
where the positive integer $m$ is chosen to give the last inequality, i.e., $(2\gamma+\delta_0)e^{-bm} \leq \gamma$. Therefore $x \sim_h y$, and this shows that $(X,E)$ is $(m,1)$-departing.  The hyperbolicity of $AI_b$-graph follows from Theorem \ref{th2.11}.

\vspace {0.1cm}

By Proposition \ref{th3.1} (with $k=1$) and \eqref{eq4.6}, we see that  two rays ${\bf x},{\bf y}$ are equivalent if and only if ${\rm dist} (\Phi(x_i),\Phi(y_i)) \leq \gamma  e^{-bi}$ for all $i$. This verifies \eqref{eq4.3} so that $\kappa : \partial X \rightarrow K$ is a well-defined bijection,  and hence  $(X,E^{(b)})$ is an  admissible graph.

\vspace {0.1cm}
We now prove that $\kappa$ is a H\"older equivalence. For distinct $\xi,\eta \in \partial X$,
 we take two rays ${\bf x},{\bf y} \in \mathcal R_v$ that converge to $\xi,\eta$ respectively with $(\xi|\eta) = ({\bf x}|{\bf y})$ (by \eqref{eq3.8} and the following remark). Let $n = |{\bf x} \vee {\bf y}|_1$ as in \eqref{eq3.3} with $k=1$, i.e., $d_h(x_n,y_n) \leq 1$ and $d_h(x_{n+1},y_{n+1}) \geq 2$. By Proposition \ref{th3.1}, we have $|(\xi|\eta)-n| = |({\bf x}|{\bf y})-n| \leq D_0$ for some $D_0 := D_0(m,1) > 0$. As $\kappa(\xi) \in \Phi (x_{n+1}) \subset \Phi (x_n)$ and $\kappa(\eta) \in \Phi (y_{n+1}) \subset \Phi(y_n)$, we get the lower bound of \eqref{eq4.7} by
\begin{align*}
\rho(\kappa(\xi),\kappa(\eta)) &\geq {\rm dist} (\Phi(x_{n+1}),\Phi(y_{n+1})) \\
&\geq \gamma e^{-b(n+1)} \geq \gamma e^{-b(D_0+1)}e^{-b(\xi|\eta)} \geq C_1\theta_a(\xi,\eta)^{b/a},
\end{align*}
and the upper bound by
\begin{align*}
\rho(\kappa(\xi),\kappa(\eta)) &\leq |\Phi(x_n)| +{\rm dist}(\Phi(x_n),\Phi(y_n)) + |\Phi(y_n)|\\
&\leq (2\delta_0+\gamma)e^{-bn} \leq (2\delta_0+\gamma) e^{bD_0}e^{-b(\xi|\eta)} \leq C_2 \theta_a(\xi,\eta)^{b/a}.
\end{align*}
This completes the proof.
\end{proof}

\bigskip

Now we turn to study the $AI_\infty$-graphs. Unlike the $AI_b$-graph, the $AI_\infty$-graph is not always hyperbolic (see Example \ref{ex6.1}), and  sufficient conditions for its hyperbolicity will be provided in Section \ref{sec:5}. The following result shows that if it is hyperbolic, then the admissibility follows.

\begin{proposition} \label{th4.6}
Suppose the $AI_\infty$-graph $(X,E^{(\infty)})$ is hyperbolic. Then for two rays ${\bf x},{\bf y} \in \mathcal R_v$, the following assertions are equivalent.

\vspace{0.1cm}
(i) \ ${\bf x}$ and ${\bf y}$ are equivalent; \quad (ii) \ $\kappa_0({\bf x}) = \kappa_0({\bf y})$; \quad (iii) \ $d_h(x_i,y_i) \leq 1$ for all $i \geq 0$.

\vspace {0.1cm}
\noindent It follows from \eqref {eq4.3} that $(X,E^{(\infty)})$ is an admissible graph (Definition \ref {de4.2}).
 \end{proposition}

\begin{proof}
(iii) $\Rightarrow$ (i) is clear.
For (ii) $\Rightarrow$ (iii), since $\kappa_0({\bf x}) \in \Phi(x_i)$ and $\kappa_0({\bf y}) \in \Phi(y_i)$, we have $\Phi(x_i) \cap \Phi(y_i) \neq \emptyset$ for all $i \geq 0$. This yields (iii) by the definition \eqref{eq4.5} of $E_h^{(\infty)}$.

\vspace{1mm}

For (i) $\Rightarrow$ (ii), as $(X,E^{(\infty)})$ is expansive and hyperbolic, by Theorem \ref{th2.11}(iii) and Proposition \ref{th3.1}, there exists an integer $k>0$ such that $d_h(x_i,y_i) \leq k$ for all $i \geq 0$. We show inductively that any such $k$ will imply (ii). When $k = 1$, it follows that $\Phi(x_i) \cap \Phi(y_i) \neq \emptyset$ for all $i \geq 0$. By the compactness of $\Phi(x_i) \cap \Phi(y_i)$, the intersection $\big({\bigcap}_{i=0}^\infty \Phi(x_i)\big) \cap \big({\bigcap}_{i=0}^\infty \Phi(y_i)\big) = {\bigcap}_{i=0}^\infty (\Phi(x_i) \cap \Phi(y_i)) \neq \emptyset$, hence $\kappa_0({\bf x}) = \kappa_0({\bf y})$.

\vspace{1mm}

Inductively, suppose (ii) is verified when $k = n$ for some $n > 0$. Let ${\bf x},{\bf y} \in \mathcal R_v$ satisfy $d_h(x_i,y_i) \leq n+1$. Then for each $i \geq 0$, we choose $z_i \in X$ such that $d_h(x_i,z_i) \leq 1$ and $d_h(z_i,y_i) \leq n$. The sequence $\{z_i\}_{i=0}^\infty$ may not be a ray; however, using the local finiteness of $(X,E_v)$ and a diagonal argument, we can choose a ray ${\bf w} \in \mathcal R_v$ such that each $w_i$ contains an infinite subsequence of $\{z_i\}_{i=0}^\infty$, and the expansive property \eqref{eq2.2} implies that $d_h(x_i,w_i) \leq 1$ and $d_h(w_i,y_i) \leq n$ for all $i \geq 0$. Using the induction hypothesis, we have $\kappa_0({\bf x}) = \kappa_0({\bf w}) = \kappa_0({\bf y})$, and the proof is completed by induction.
\end{proof}

When $\Phi$ is of exponential type-$(b)$, comparing the $AI_\infty$-graph with the $AI_b$-graph, it is clear that $E^{(\infty)} \subset E^{(b)}$. The following is a consequence of Theorem \ref{th4.5}.

\begin{corollary} \label{th4.7}
Suppose the index map $\Phi$ is of exponential type-$(b)$, and the associated $AI_\infty$-graph $(X,E^{(\infty)})$ is hyperbolic. Then $\kappa: (\partial X,\theta_a) \to (K,\rho)$ is H\"older continuous  with exponent $b/a$, i.e.,
\begin{equation} \label{eq4.8}
\rho(\kappa(\xi),\kappa(\eta))^{a/b} \leq C \theta_a(\xi,\eta) , \qquad \forall\ \xi,\eta \in \partial X.
\end{equation}
\end{corollary}

\begin{proof} 
We consider the associated $AI_b$-graph $(X,E^{(b)})$, and denote its graph distance and Gromov product by $d'(\cdot,\cdot)$ and $(\cdot|\cdot)'$ respectively.
By Theorem \ref{th4.5}, $(X,E^{(b)})$ is hyperbolic,  and the bijection $\kappa': \partial X' \to K$ satisfies $\rho(\kappa'(\xi), \kappa'(\eta)) \asymp e^{-b(\xi|\eta)'}$ for all $\xi,\eta \in \partial X'$.
As $\kappa_0 = \kappa_0'$ on $\mathcal R_v$, it follows that $\partial X = \partial X'$ and $\kappa = \kappa'$.

\vspace{1mm}

Note that $E^{(\infty)} \subset E^{(b)}$. This implies $d(x,y) \geq d'(x,y)$, and
$$
(x|y) = \frac 12(|x|+|y|-d(x,y)) \leq \frac 12(|x|+|y|-d'(x,y)) = (x|y)'
$$
for all $x,y \in X$. Taking the limits in \eqref{eq3.8}, we have $(\xi|\eta) \leq (\xi|\eta)'$, and
$$
 \theta_a(\xi,\eta) \geq c_1 e^{-a(\xi|\eta)} \geq c_1 e^{-a(\xi|\eta)'} \geq  c_2 \rho(\kappa'(\xi),\kappa'(\eta))^{a/b} = c_2 \rho(\kappa(\xi),\kappa(\eta))^{a/b}.
$$
This verifies \eqref{eq4.8},  the H\"older continuity of $\kappa$. 
\end{proof}

\medskip

In general, we cannot expect this $\kappa$ to be a H\"older equivalence (see Example \ref {ex6.2}). In the following theorem,  we give a characterization for the hyperbolicity, or equivalently the $(m,k)$-departing property, of the $AI_\infty$-graph associated to a saturated index map (Remark 2 of Definition \ref {de4.1}), together with the H\"older equivalence to hold.

\medskip

\begin{theorem} \label{th4.8}
Suppose $\Phi$ is a saturated index map on $(X,E_v)$ over $(M,\rho)$. Then for $b \in (0,\infty)$ and an integer $k>0$, the following assertions are equivalent.

\vspace{1mm}

\ (i) \ The $AI_\infty$-graph $(X,E^{(\infty)})$ is $(m,k)$-departing for some $m>0$,
and $\kappa: (\partial X,\theta_a) \to (K,\rho)$ is a H\"older equivalence with exponent $b/a$, i.e.,
\begin{equation} \label{eq4.10}
\rho(\kappa(\xi),\kappa(\eta))^{a/b} \asymp  \theta_a(\xi,\eta), \qquad \forall\ \xi,\eta \in \partial X.
\end{equation}

(ii) \ $\Phi$ is of exponential type-$(b)$ under $\rho$, and there exists $\gamma>0$ such that $(X,E^{(\infty)})$ satisfies for $x, y \in X$,
\begin{equation} \label{eq4.11}
|x| = |y| \ \ \hbox{and} \ \ d_h(x,y) > k \quad \Rightarrow \quad {\rm dist}_\rho(\Phi(x),\Phi(y)) > \gamma e^{-b|x|}.
\end{equation}
\end{theorem}

\begin{proof}
For (i) $\Rightarrow$ (ii), since $(X,E^{(\infty)})$ is expansive and $(m,k)$-departing, it is hyperbolic (Theorem \ref{th2.11}). By Proposition \ref{th4.6}, it is an admissible graph, and the bijection $\kappa: \partial X \to K$ is well-defined. Using \eqref{eq4.4} and \eqref{eq4.10}, we have
$$
|\Phi(x)|_\rho = |\kappa(\mathcal J_\partial(x))|_\rho \leq C_1 |\mathcal J_\partial(x)|_{\theta_a}^{b/a} \leq C_1 e^{-b|x|}
$$
for all $x \in X$. Therefore $\Phi$ is of exponential type-$(b)$ under $\rho$.
Moreover, for $x,y \in X_n$ with $d_h(x,y) > k$, it follows from \eqref{eq4.4}, \eqref{eq4.10} and Proposition \ref{th3.2} that
\begin{align*}
{\rm dist}_\rho(\Phi(x),\Phi(y)) &= {\rm dist}_\rho(\kappa(\mathcal J_\partial(x)),\kappa(\mathcal J_\partial(y))) \\
&\geq C_2 \,{\rm dist}_{\theta_a}(\mathcal J_\partial(x),\mathcal J_\partial(y))^{b/a} > C_2 \gamma_0 e^{-bn}.
\end{align*}
 This proves \eqref{eq4.11} with $\gamma = C_2 \gamma_0$.

\vspace{1mm}

For (ii) $\Rightarrow$ (i), the proof is similar to the one of Theorem \ref{th4.5} on the $AI_b$-graph.
Let $\delta_0 = \sup_{z \in X} e^{b|z|}|\Phi(z)|_\rho$, which is finite as $\Phi$ is of exponential type-$(b)$. Let $u \in \mathcal J_m(x)$ and $v \in \mathcal J_m(y)$ with $d_h(u,v)=\ell<2k$. Then $\Phi(u)$ and $\Phi(v)$ are joined by a chain $\{\Phi(u_j)\}_{j=0}^\ell$ with $u_0=u$ and $u_\ell=v$, in which $|u_j| = |u| = |x|+m$ and $\Phi(u_{j-1}) \cap \Phi(u_j) \neq \emptyset$ for all $j \in \{1,2,\cdots,\ell\}$. Therefore,
$$
{\rm dist}_\rho(\Phi(x),\Phi(y)) \leq {\rm dist}_\rho(\Phi(u),\Phi(v)) \leq {\sum}_{j=1}^{\ell-1} |\Phi(u_j)|_\rho \leq (2k-1)\delta_0 e^{-b(|x|+m)} \leq \gamma e^{-b|x|},
$$
where the integer $m>0$ is chosen to satisfy the last inequality, i.e., $(2k-1)\delta_0 e^{-bm} \leq \gamma$. By \eqref{eq4.11}, we have $d_h(x,y) \leq k$, and this proves that $(X,E^{(\infty)})$ is $(m,k)$-departing.

\vspace{1mm}

To prove \eqref{eq4.10}, for $\xi \neq \eta \in \partial X$, we take ${\bf x},{\bf y} \in \mathcal R_v$ that converge to $\xi,\eta$ respectively with $(\xi|\eta) = ({\bf x}|{\bf y})$, and let $n=|{\bf x}\vee {\bf y}|_k$ as in \eqref{eq3.3}. It follows from Proposition \ref{th3.1} that $|(\xi|\eta)-n| = |({\bf x}|{\bf y})-n| \leq D_0$ for some $D_0 = D_0(m,k) > 0$. Using $\kappa(\xi) \in \Phi (x_{n+1})$, $\kappa(\eta) \in \Phi (y_{n+1})$, $d_h(x_{n+1},y_{n+1}) > k$ and \eqref{eq4.11}, we get the lower bound of \eqref{eq4.10} by
\begin{align*}
\rho(\kappa(\xi),\kappa(\eta)) &\geq {\rm dist}_\rho (\Phi(x_{n+1}),\Phi(y_{n+1})) \\
&> \gamma e^{-b(n+1)} \geq \gamma e^{-b(D_0+1)}e^{-b(\xi|\eta)} \geq c \theta_a(\xi,\eta)^{b/a};
\end{align*}
the upper bound is proved by Corollary \ref{th4.7}. This completes the proof.
\end{proof}

\medskip

\noindent {\bf Remark.}
Letting $k=1$ in \eqref{eq4.11}, the condition becomes
\begin{equation} \label{eq4.13}
|x|=|y|\quad \hbox{ and }\quad \Phi(x) \cap \Phi(y) = \emptyset \quad \Rightarrow  \quad {\rm dist}  (\Phi(x),\Phi(y)) > \gamma e^{-b|x|}.
\end{equation}
This is the {\it condition (H)} in \cite{LW1}  in the setup where $K$ is a self-similar set (see also Appendix, Theorem \ref{th7.1}); the authors proved that this condition is sufficient for the H\"older equivalence between $\partial X$ and $K$. Here a necessity part is also provided, together with an extra relation with the $(m,k)$-departing property (or hyperbolicity) of these graphs.

\bigskip

\section{Separation conditions and doubling metrics}
\label{sec:5}

In this section,  we aim to give some sufficient conditions for $AI_\infty$-graphs to be hyperbolic, and characterize the bounded degree property of $AI_b$- and $AI_\infty$-graphs;
both involve some separation conditions on index maps and the doubling property of attractors.

\medskip

Let $\Phi$ be an index map on a vertical graph $(X, E_v)$ over a complete metric space $(M, \rho)$ with attractor $K$.  We call a map $\iota: X \to K$ a {\it projection}  (with respect to $\Phi$) if it satisfies $\iota (x) \in K_x$($:=\Phi(x) \cap K$) for all $ x\in X$. When there is no confusion, we shall denote the ball $B_\rho(\xi,r)$ by $B(\xi,r)$ for simplicity.

\medskip

\begin{definition} \label{de5.1}
Let $\Phi$ be an index map with attractor $K$.
 For $b \in (0, \infty)$, we say that $\Phi$ (or $\{K_x\}_{x\in K}$) satisfies
\begin{enumerate}[(i)]
\item  {\rm condition ($S_{b}$)}  if for any $c>0$, there is a constant $\bar\ell = \bar\ell(c)$ such that
\begin{equation} \label{eq5.1} 
\#\{x \in X_n: K_x \cap F \neq \emptyset\} \leq \bar\ell, \quad \forall \ n \geq 0 \ \hbox{and} \ F \subset M \ \hbox{with} \ |F|_\rho < ce^{-bn};
\end{equation}
\item  {\rm condition ($B_{b}$)}  if there exist a projection $\iota: X \to K$ and $c_0 \in (0,\infty)$ such that
\begin{equation} \label{eq5.2} 
B(\iota(x), c_0e^{-b|x|}) \cap K \subset K_x, \qquad \forall\ x \in X.
\end{equation}
\end{enumerate}
\end{definition}

We remark that condition $(S_b)$ is motivated from the open set condition on self-similar sets (see \eqref{eq7.3} in Appendix, and condition $(S_b')$ in Proposition \ref{th5.3}).
To study the above conditions, we need a preliminary result on doubling metric spaces (Lemma \ref{th5.2}). For a subset $F \subset M$ and $r > 0$, define the {\it $r$-covering number} of $F$ by
\begin{equation} \label{eq5.3}
N_r^c(F) := \inf\{\# \Xi:\  \Xi \subset F \hbox{ and } F \subset {\bigcup}_{\xi \in \Xi} \,B(\xi,r)\};
\end{equation}
$F$ is called {\it totally bounded} if the $r$-covering number is finite for all $r>0$.
We also define the {\it $r$-packing number} of $F$ to be
\begin{align*} 
N^p_r(F) := \sup\{\# \Xi: \ \Xi \subset F \hbox{ and } B(\xi,r) \cap B(\eta,r) = \emptyset, \ \forall\ \xi,\eta \in \Xi,\ \xi \neq \eta\},
\end{align*}
and the
{\it $r$-separating number} of $F$ to be
$$
N_r^s(F):= \sup\{\# \Xi: \Xi \subset F \hbox{ and } \rho(\xi,\eta) \geq r, \ \forall\ \xi,\eta \in \Xi,\ \xi \neq \eta\}.
$$
Recall that the metric space $(M,\rho)$ is doubling (Definition \ref{de3.5}) if and only if
$$
\sup\{N^c_r(B(\xi,2r)): \ \xi \in M,\, r>0\}<\infty.
$$

\begin{lemma} \label{th5.2}
Let $(M,\rho)$ be a metric space. Then the inequalities
\begin{equation} \label{eq5.5}
N_{2r}^s(F) \leq N_r^p(F) \leq N_r^c(F) \leq N_r^s(F)
\end{equation}
hold for all totally bounded subset $F \subset M$ and $r>0$. As a consequence, the following assertions are equivalent.
\begin{enumerate}[(i)]
\item $(M,\rho)$ is doubling (Definition \ref{de3.5}).

\item For some ($\Leftrightarrow$ any) $t>1$, $\widehat{N}^c(t):=\sup\{N^c_r(B(\xi,tr)): \xi \in M,\,r>0\}<\infty$.

\item For some ($\Leftrightarrow$ any) $t>1$, $\widehat{N}^p(t):=\sup\{N^p_r(B(\xi,tr)): \xi \in M,\,r>0\}<\infty$.

\item For some ($\Leftrightarrow$ any) $t>1$, $\widehat{N}^s(t):=\sup\{N^s_r(B(\xi,tr)): \xi \in M,\,r>0\}<\infty$.
\end{enumerate}
\end{lemma}

The inequality \eqref {eq5.5} is straightforward by using the definitions, and the equivalence of (i)--(iv) follows directly from \eqref {eq5.5}. We omit the detail.

\medskip

\begin{proposition} \label{th5.3}
Let $\Phi$ be an index map with attractor $K$, and is of exponential type-$(b)$.
Suppose $(K, \rho)$ is doubling, then the following conditions are equivalent:
\begin{enumerate}[(i)]
\item condition $(S_b)$;

\item condition $(S_b')$:
there exist $c_1 >0$ and $\ell_1>0$ such that
\begin{equation} \label{eq5.6}
\# \{x \in X_n: K_x \cap  B(\xi,c_1 e^{-bn})\not = \emptyset\} \leq \ell_1, \qquad \forall\ n\geq 0, \ \xi \in K;
\end{equation}

\item condition $(S_b'')$:
there exist a projection $\iota: X \to K$, $c_2>0$ and $\ell_2>0$ such that
\begin{equation} \label{eq5.7}
\# \{x \in X_n: \rho(\iota(x),\iota(y)) < c_2 e^{-bn}\} \leq \ell_2, \qquad \forall\ n\geq 0, \ y \in X_n.
\end{equation}
\end{enumerate}
\end{proposition}

\begin{proof}
(i) $\Rightarrow$ (ii) $\Rightarrow$ (iii) is obvious. We need only show (ii) $\Rightarrow$ (i) and (iii) $\Rightarrow$ (ii).

\vspace{1mm}

(ii) $\Rightarrow$ (i): For any $c>0$ and $F \subset M$ with $|F|_\rho < c e^{-bn}$, we will show that \eqref{eq5.1} holds for some $\bar{\ell}$. Without loss of generality, we assume that $c > c_1$ and $K \cap F \not= \emptyset$. Fix a point $\xi \in K \cap F$, and denote the open ball $B(\xi, c e^{-bn}) \cap K$ in $(K,\rho)$ by $B_K$. Since $(K, \rho)$ has doubling property, by Lemma \ref{th5.2}(ii), $B_K$ can be covered by a union of $N_0 := \widehat{N}^c(\frac{c}{c_1})$ open balls $B(\xi_i, c_1 e^{-bn}) \cap K$ in $(K,\rho)$, where $\xi_i \in K$, $i =1, \cdots, N_0$, i.e.,
\[
K \cap F \subset B_K \subset {\bigcup}_{i=1}^{N_0}  \Big(K \cap B(\xi_i, c_1 e^{-bn})\Big).
\]
It follows from \eqref{eq5.6} that
$$
\# \{ x \in X_n: \ K_x \cap F \not= \emptyset \} \le {\sum}_{i=1}^{N_0} \,\# \{ x\in X_n: \, K_x \cap B(\xi_i, c_1 e^{-bn}) \not= \emptyset\} \le \ell_1 N_0.
$$
We see that \eqref{eq5.1} holds for the constant $\bar{\ell} = \ell_1 N_0$.

\vspace{1mm}

(iii) $\Rightarrow$ (ii): For $n \geq 0$ and $\xi \in K$, denote
\[
X_n(\xi) := \{ x\in X_n: \, K_x \cap B(\xi, c_1 e^{-bn}) \not= \emptyset \},
\]
on which we define an edge set $\widetilde{E} = \{(x,y) \in X_n(\xi) \times X_n(\xi) \setminus \Delta: \rho(\iota(x),\iota(y)) < c_2 e^{-bn}\}$. By \eqref{eq5.7}, we see that the maximal degree in the graph $(X_n(\xi), \widetilde{E})$ does not exceed  ($\ell_2 -1$). Applying Brooks' Theorem on the chromatic number (cf.~\cite[Theorem 3.1]{LiW}), there exists a coloring map $\mathcal K: X_n(\xi) \to \Sigma$ with $\# \Sigma \leq \ell_2$ such that $\mathcal K(x) \neq \mathcal K(y)$ if $(x,y) \in \widetilde{E}$.

\vspace{1mm}

Since $\Phi$ is of exponential type-$(b)$, there exists a constant $\delta_0 >0$ such that $|\Phi(x)|_\rho \le \delta_0 e^{-b|x|}$ for all $x \in X$. Consider the discrete sets
\[
\mathcal M^{(i)} := \{\iota(x): \, x \in X_n(\xi),\ \mathcal K(x) = i\}, \qquad i \in \Sigma.
\]
Then $\rho(\iota(x),\iota(y)) \geq c_2 e^{-bn}$ for $x \neq y \in \mathcal M^{(i)}$. By the triangle inequality, it is easy to see that $\mathcal M^{(i)} \subset B(\xi, (c_1 + \delta_0) e^{-bn})$ holds for all $i \in \Sigma$.
Using Lemma \ref{th5.2}(iii), we have
$$
\# X_n(\xi) = {\sum}_{i \in \Sigma} \, \# \mathcal M^{(i)} \leq \ell_2 N^s_{c_2 e^{-bn}}(B(\xi, (c_1 + \delta_0) e^{-bn}) \cap K) \leq \ell_2\widehat{N}^s((c_1 + \delta_0) / c_2)<\infty.
$$
This completes the proof by letting $\ell_1 = \ell_2\widehat{N}^s((c_1 + \delta_0) / c_2)$.
\end{proof}

\medskip

\medskip

The following theorem provides sufficient conditions for $AI_\infty$-graphs to be hyperbolic, and completes the study of $AI_\infty$-graphs in the last section.

\medskip

\begin{theorem} \label{th5.4}
Let $\Phi$ be an index map with  attractor $K$, and is of exponential type-$(b)$. If either
\begin{enumerate}[(i)]
\item condition $(S_{b})$ is satisfied; or

\item the attractor $(K,\rho)$ is doubling, and condition $(B_{b})$ is satisfied,
\end{enumerate}
then the $AI_\infty$-graph is hyperbolic, and hence an admissible graph.
\end{theorem}

\begin{proof}
We prove the hyperbolicity of $(X,E^\infty)$ by using a similar argument as in \cite[Theorem 1.2]{LW3}.  Suppose it is not hyperbolic.  Then by Theorem \ref{th2.11}, for any integer $m>0$, there exists a horizontal geodesic $[x_0,x_1,\cdots,x_{3m}]$ with length $3m$. Clearly $|x_0| =:n > m$, and by the expansive property \eqref{eq2.1}, there is a horizontal path $[y_0,y_1,\cdots,y_\ell]$ in $X_{n-m}$  (see Figure \ref{fig:5}) such that
\begin{enumerate}[(a)]
\item $y_0 \in \mathcal J_{-m}(x_0)$, $y_\ell \in \mathcal J_{-m}(x_{3m})$, and $y_i \in \bigcup_{j=0}^{3m} \mathcal J_{-m}(x_j)$ for all $0<i<\ell$;

\item $(y_i,y_j) \in E_h$ if and only if $|i-j|=1$.
\end{enumerate}

\begin{figure}[ht]
\begin{center}
\includegraphics[scale=0.5]{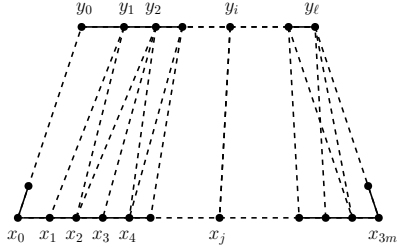}
\caption{Two horizontal paths $[x_j]_j$ and $[y_i]_i$.} \label{fig:5}
\end{center}
\end{figure}

\vspace{1mm}

As $3m = d(x_0, x_{3m}) \leq d(x_0,y_0)+d(y_0,y_\ell)+d(y_\ell,x_{3m}) \leq m+\ell+m$, we have $\ell \geq m$. Let $F:=\bigcup_{i=0}^\ell \Phi(y_i)$. Note that by (a), for any $\eta \in F$, there exists $k \in \{0,1,\ldots,3m\}$ such that both $\Phi(x_k)$ and $\eta$ are contained in some $\Phi(y_i)$. Thus for any projection $\iota: X \to K$,
\begin{align*}
\rho(\iota(x_0), \eta)
&\leq {\sum}_{j=0}^{k-1}\,\rho(\iota(x_j),\iota(x_{j+1})) + |\Phi(y_i)| \\
&\leq 2k\delta_0 e^{-bn} + \delta_0 e^{b(m-n)} \leq \delta_0(6me^{-bm}+1) e^{b(m-n)},
\end{align*}
where $\delta_0 := \sup_{z \in X} e^{b|z|}  |\Phi(z)| < \infty$.
Take $m$ large enough such that $6me^{-bm} < 1$. Then
$$
F = {\bigcup}_{i=0}^\ell \Phi(y_i)\subset B(\iota(x_0),2\delta_0 e^{b(m-n)})=:B.
$$

\vspace{2mm}

(i) If condition (S$_{b}$) is satisfied, by noting that $|F| \leq 4\delta_0 e^{b(m-n)} =: ce^{b(m-n)}$,  we have
$$
\ell+1 \leq \#\big \{y \in X_{n-m}: \Phi(y) \subset F\big \} \leq  \#\big \{y \in X_{n-m}: K_y \cap F \neq \emptyset\big \}\leq \bar\ell(c).
$$
This is impossible since $\ell \geq m$ can be arbitrarily large.

\vspace{2mm}

(ii) If $(K,\rho)$ is doubling and condition ($B_{b}$) holds, then we choose the above $\iota$ to satisfy \eqref{eq5.2}. 
By (b), the ball $B$ contains at least $\lfloor\ell/2\rfloor+1>m/2$ mutually disjoint balls $B(\iota(y_{2i}), c_0e^{b(m-n)})$, $0\leq i \leq \lfloor\ell/2\rfloor$. On the other hand by considering the packing number and the constant of doubling (see Lemma \ref{th5.2}(iii)), we have
$$
N_{e^{b(m-n)}}^p(B \cap K) \leq \widehat{N}^p (2\delta_0 / c_0) < \infty.
$$
As $m$ can be arbitrarily large, this is impossible.
Hence $(X,E^\infty)$ is hyperbolic in either case. By  Proposition \ref{th4.6}, the $AI_\infty$-graph is an admissible graph.
\end{proof}

\bigskip

Recall that for an expansive hyperbolic graph with bounded degree, the hyperbolic boundary possesses the doubling property (Theorem \ref{th3.6}).  With the separation property,  we have a stronger result for $AI_b$-graphs as well as a sufficient condition for $AI_\infty$-graphs.

\medskip

\begin{theorem} \label{th5.5}
Let $\Phi$ be an index map of exponential type-$(b)$. Then the $AI_b$-graph has bounded degree if and only if condition $(S_b)$ is satisfied. As a consequence, $(S_b)$ condition  is sufficient for the $AI_\infty$-graph to have bounded degree.
\end{theorem}

\begin{proof}
It is clear that the last statement follows from the first one, since $E^\infty \subset E^{(b)}$.

\vspace{1mm}

To show the necessity of the first part, by Theorem \ref{th4.5}, the $AI_{b}$-graph is hyperbolic, and the hyperbolic boundary $(\partial X,\theta_a)$ is H\"{o}lder equivalent to the attractor $(K,\rho)$. From Theorem \ref{th3.6} we know that $(\partial X,\theta_a)$ has the doubling property, which is preserved by the H\"older equivalence $\kappa$, and thus $(K,\rho)$ is also doubling.

Next we will show that the index map $\Phi$ satisfies condition $(S_b'')$. We fix an arbitrary projection $\iota$ from $X$ to $K$. For $x,y \in X_n$ with $\rho(\iota(x),\iota(y)) < \gamma e^{-bn}$ (where $\gamma>0$ is the constant in \eqref{eq4.6}, the definition of $E_h^{(b)}$), we have $x \sim_h y$ since
${\rm dist}_\rho(\Phi(x), \Phi(y)) \leq \rho(\iota(x),\iota(y)) < \gamma e^{-bn}$.
Hence for all $y \in X_n$,
$$
\# \{x \in X_n: \rho(\iota(x),\iota(y)) < \gamma e^{-bn}\} \leq \# \{x \in X_n: x \sim_h y\} \leq \deg(y).
$$
As $\ell_2:=\sup_{y \in X} \deg(y) < \infty$, \eqref{eq5.7} holds for $c_2 = \gamma$. Making use of Proposition \ref{th5.3}, we see that condition $(S_b)$ is satisfied.

\vspace{2mm}

For the sufficiency, we calculate the degree of a fixed vertex $y \in X_n$. Set $\delta_0:=\sup_{z \in X} e^{b|z|} |\Phi(z)|$. For $x \in \mathcal J_1(y)$, we see that $K_x \subset \Phi(x) \subset \Phi(y)$. As $|\Phi(y)| \leq \delta_0 e^{-bn} =: ce^{-b(n+1)}$, condition $(S_b)$ implies that
$$
\# \mathcal J_1(y) \leq \# \{x \in X_{n+1}: K_x \cap \Phi(y) \neq \emptyset\} \leq  \bar{\ell}(c).
$$
For $x \in \mathcal J_{-1}(y)$, it follows that
$$
\Phi(y) \subset \Phi(x) \subset F:=\{\xi \in M: \rho(\xi,\Phi(y)) \leq \delta_0 e^{b(1-n)}\}.
$$
Using the triangle inequality, we get $|F| < 3\delta_0 e^{b(1-n)}=: c'e^{b(1-n)}$,  and thus by condition $(S_b)$,
$$
\# \mathcal J_{-1}(y) \leq \# \{x \in X_{n-1}: K_x \cap F \neq \emptyset\} \leq \bar{\ell}(c').
$$
For $x \in X_n$ with $(x,y) \in E_h$, using the triangle inequality we have
$$
K_x \subset \Phi(x) \subset G:=\{\xi \in M: \rho(\xi,\Phi(y)) \leq (\gamma+\delta_0) e^{-bn}\}.
$$
It follows in a similar way that $|G| < 3(\gamma + \delta_0)e^{-bn}=:c''e^{-bn}$, and hence
$$
 \# \{x \in X: (x,y) \in E_h\} \leq \# \{x \in X_n: K_x \cap G \neq \emptyset\} \leq \bar{\ell}(c'').
$$
As $\deg(y) = \#\mathcal J_{-1}(y) + \#\mathcal J_1(y) + \# \{x \in X: (x,y) \in E_h\}$, we conclude from the above estimates that $\deg(y)$ is uniformly bounded by $\bar{\ell}(c)+\bar{\ell}(c')+\bar{\ell}(c'')$ for all $y \in X$, so that the $AI_b$-graph is of bounded degree. We complete the proof.
\end{proof}

 As a consequence of Theorems \ref{th3.6}, \ref{th5.5} and  Proposition \ref{th5.3}, we have

\begin{corollary}  \label{th5.6}
Let $\Phi$ be an index map with attractor $K$, and is of exponential type-$(b)$. Then the $AI_{b}$-graph has bounded degree if and only if the attractor $(K,\rho)$ is doubling, and condition $(S_b'')$ (or $(S_b')$) in Proposition \ref{th5.3} is satisfied.
\end{corollary}

\bigskip

\section{Examples and more on IFSs}
\label{sec:6}

We first give an example that the $AI_\infty$-graph is not hyperbolic.

 \begin{example} \label{ex6.1} \hspace{-2mm} {\bf (An anisotropic binary partition of unit square)}
{\rm We define  a binary subdivision scheme to partition  $[0,1]^2$  iteratively into subrectangles. By a Type-I subdivision, we mean dividing a rectangle $J$ horizontally into two equal subrectangles; Type-II means dividing $J$ vertically instead. Let ${\mathcal L} = \{\frac 12 \ell(\ell+ 1): \hbox{integer } \ell >0\} = \{ 1, 3, 6 \cdots\}$. Now let $J_\vartheta = [0,1]^2$. Suppose we have defined $ J_x$ with $x =x_1 \cdots x_n,\ x_i = 0,1$.  If $n+1 \in \mathcal L$, we use Type-I subdivision on  $J_x$ to obtain $J_{x0}$ and $J_{x1}$; otherwise, we use Type-II for the subdivision (see Figure \ref{fig:7}).

\begin{figure}[ht]
\begin{center}
\includegraphics[height=2.0cm]{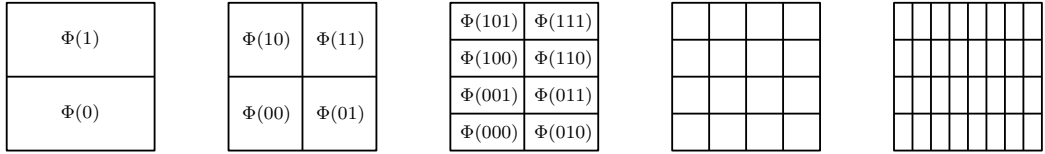}
\caption{The binary partition $\{\Phi(x)\}_{x \in X_n}$, $n=1,2,3,4,5$;  \  $1, 3 \in {\mathcal L}$. } \label{fig:7}
\end{center}
\end{figure}

Let $(X,E_v)$ be the corresponding  {\it binary tree}, on which we define the index map $\Phi(x) = J_x$ and consider the $AI_\infty$-graph $(X,E)$. Fix an integer $\ell>0$. Let $n := \frac{\ell(\ell+1)}{2}$, $x = 0^n$ and
$
y=1010^21 \cdots 0^{\ell-1}1.
$
Then $\Phi(x) = [0,2^{-n+ \ell}] \times [0,2^{-\ell}]$ and $\Phi(y)=[0,2^{-n+\ell}] \times [1-2^{-\ell}, 1]$, which are the rectangles in the lower-left and the upper-left corners of $[0,1]^2$ respectively. It is clear that $d_h(x,y) = 2^\ell-1$. Taking $u = 0^{n+\ell} \in \mathcal J_\ell(x)$ and
$
v= 1010^21 \cdots 0^{\ell-1}10^\ell \in \mathcal J_\ell(y),
$
we can also check that $d_h(u,v) = d_h(x,y) = 2^\ell-1$. This shows that $(X,E)$ is not $(m,k)$-departing whenever $m \leq \ell$ and $k \leq 2^\ell-2$. As $\ell$ can be arbitrary, $(X,E)$ is not hyperbolic by Theorem \ref{th2.11}.
\hfill $\square$}
\end{example}

 \medskip

 Our next example gives an IFS that is homogeneous (the contraction ratios are equal) and satisfies the OSC. The associated  $AI_\infty$-graph is hyperbolic (by Theorem \ref{th5.4}(i)), but the hyperbolic boundary is not H\"older equivalent to the attractor. This shows that (unlike the $AI_b$-graphs) the one-sided H\"older inequality in  Corollary \ref {th4.7} cannot be improved.

\medskip

\begin{example} \label{ex6.2} {\rm
In $M=\mathbb R^2$, let $p_1=(0,0)$, $p_2=(1,0)$, $p_3= (3,0)$, $p_4= (\eta,2)$, $p_5=(0,3)$ and $p_6=(3,3)$, where
$$
\eta :={\sum}_{\ell=0}^\infty (4^{-n_\ell}+4^{-n_\ell-1}) \quad \hbox{with} \quad n_\ell = 1+\frac{\ell(\ell+7)}{2}, \ \ell=0,1,\cdots
$$
Let $S_j(x)=\frac 14 (x+p_j)$, $j \in \Sigma = \{1,2,\cdots,6\}$, and let $K$ be the self-similar set of the IFS $\{S_j\}_{j=1}^6$. We set $X = \Sigma^*$ (the symbolic space) and $\Phi(x) = S_x([0,1]^2)$, $x \in X$. Clearly ${\bigcup}_{j=1}^6 \Phi(j) \subset [0,1]^2$, and every $\Phi(x)$ is a square with side length $4^{-|x|}$ (see Figure \ref{fig:6}). Hence $\Phi$ is an index map of exponential type-$(b)$ with $b=\ln 4$, and it is easy to see that the associated saturated map $\widetilde\Phi(x) = S_x(K)$ for any $x \in X$ (see \eqref {eq4.1}) .

\begin{figure}[ht]
\begin{center}
\includegraphics[scale=0.6]{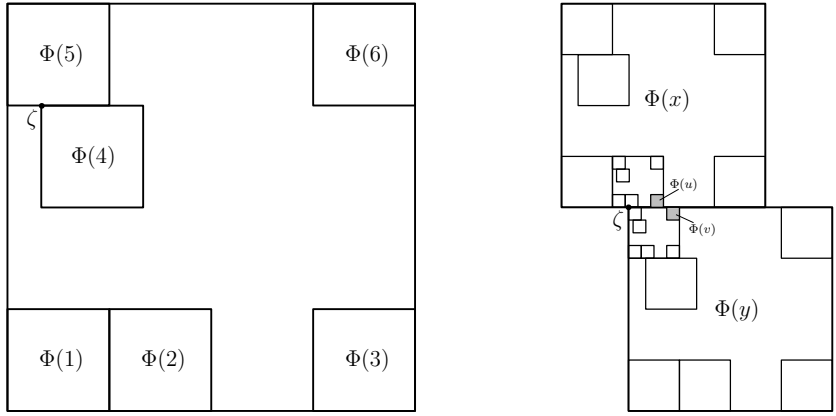}
\caption{The squares in Example \ref{ex6.2}.} \label{fig:6}
\end{center}
\end{figure}

Fix an integer $\ell > 0$. Let $x=5 221^2 221^3 \cdots 221^{\ell+1}$ and $y=4 5^{n_\ell-1}$. Clearly $|x| = |y| = n_\ell$, and $\Phi(y)$ has the same upper-left corner $\zeta=(\eta/4, 3/4)$ as $\Phi(4)$. Moreover, we can calculate that the lower-left corner of $\Phi(x)$ is $(\xi/4, 3/4)$, where
$$
\xi:={\sum}_{k=0}^{\ell-1} (4^{-n_k}+4^{-n_k-1}) = \eta - {\sum}_{k=\ell}^{\infty} (4^{-n_k}+4^{-n_k-1}).
$$

Let $u= x23$ and $v= y56$. Then $|u|=|v|=n_\ell+2$, the lower-right corner of $\Phi(u)$ is $(\xi/4 + 2\cdot 4^{-n_\ell-1}, 3/4)$, and the upper-left corner of $\Phi(v)$ is $(\eta/4+4^{-n_\ell-1}-4^{-n_\ell-2}, 3/4)$. It follows that $S_u(K) \cap S_v(K) = \Phi(u) \cap \Phi(v) = \emptyset$. By checking the two nearest  corners of $\Phi(u)$ and  $\Phi(v)$, we see that ${\rm dist}(\widetilde \Phi(u),\widetilde\Phi(v)) = {\rm dist}(\Phi(u),\Phi(v))$, hence
\begin{align} \label{eq6.1}
{\rm dist}(S_u(K),S_v(K)) &= {\rm dist}(\Phi(u),\Phi(v)) \nonumber \\
&= (\eta/4+4^{-n_\ell-1}-4^{-n_\ell-2}) - (\xi/4 + 2\cdot 4^{-n_\ell-1}) \nonumber \\
&= {\sum}_{k=\ell+1}^\infty (4^{-n_k-1}+4^{-n_k-2}) < 4^{-n_{\ell+1}} = 4^{-\ell-2} e^{-b|u|}.
\end{align}
Consider the $AI_\infty$-graph $(X,E)$ associated to $\widetilde\Phi$, i.e.,
$$
E_h = \{(x,y) \in X \times X \setminus \Delta: |x|=|y|,\ S_x(K) \cap S_y(K) \neq \emptyset\}.
$$
As the IFS $\{S_j\}_{j=1}^6$ satisfies the OSC, so that condition $(S_b)$ in Definition \ref{de5.1} is satisfied, by Theorem \ref {th5.4}, $(X,E)$ is hyperbolic. Since neither $\Phi(u)$ nor $\Phi(v)$ intersects other cells in $\{\Phi(z)\}_{z \in X_{n_\ell+2}}$, we have $d_h(u,v) = \infty$. As $\ell$ can be arbitrarily large, \eqref{eq6.1} implies that the condition \eqref{eq4.11} fails for any $k \geq 1$. By Theorem \ref{th4.8}, the hyperbolic boundary $(\partial X, \theta_a)$ of $(X,E)$ is not H\"older equivalent to $K$.
\hfill $\square$}
\end{example}

\bigskip

We now return to the setup in which $\{S_j\}_{j=1}^N$ is a  contractive IFS on a complete metric space $(M, \rho)$ (see Appendix). In the previous studies, the augmented tree was established according to the geometric sizes of $K_x$ \cite{LW1,LW3}. Within the framework of augmented index graphs (Section \ref{sec:4}), we are able to extend the consideration to some weighted IFSs. In the rest of this section, we are given a vector of weights  ${\bf s} \in (0,1)^N$ on $\{S_j\}_{j=1}^N$ (instead of contraction ratio ${\bf r} =(r_j)_{j=1}^N$), and expect some new metric induced on $K$.
\medskip

 {\rm
 Let $\Sigma^*$ be the  symbolic space of a contractive IFS $\{S_j\}_{j=1}^N$.
Let ${\bf s}=(s_1,s_2,\cdots,s_N)$ be a vector of weights on $\Sigma$ with $s_j \in (0,1)$
(not necessarily probability weight).
Write $s_* := \min_{j \in \Sigma}\{s_j\}$, and $s_x := s_{i_1}s_{i_2}\cdots s_{i_m}$ for $x = i_1i_2\cdots i_m \in \Sigma^*$ ($s_\vartheta = 1$ by convention). We consider a regrouping on $\Sigma^*$ by setting $X_0 := \{\vartheta\}$, and for $n \geq 1$,
\begin{equation} \label{eq6.2}
X_n = X_n({\bf s}) := \{x=i_1i_2 \cdots i_m \in \Sigma^*: \, s_x \leq s_*^n < s_{i_1}s_{i_2}\cdots s_{i_{m-1}}\}.
\end{equation}
Let $X = X({\bf s}) := \bigcup_{n=0}^\infty X_n$ denote the {\it modified coding space} with respect to ${\bf s}$.
This $X$ has a natural tree structure $E_v$ that consists of edges between each $x =i_1i_2\cdots i_m\in X_n$ and $y = i_1i_2 \cdots i_m i_{m+1} \cdots i_k \in X_{n+1}$. Let $\Phi(x) = S_x(K)$, $x \in X$. Then $\Phi$ is a saturated index map on $(X, E_v)$ over $(K, \rho)$.

\vspace{1mm}

Let $r^* = \max_{j \in \Sigma}\{r_j\}$ be the maximal contraction ratio of $\{S_j\}_{j \in \Sigma}$.
Then it is clear that $\{\Phi (x)\}_{x \in X(\bf s)}$ is always of exponential type-$(b)$ with $b=|\log r^*|$.
 From \eqref{eq4.9} and Theorem \ref{th4.5}}, the associated $AI_b$-graph $(X({\bf s}),E^{(b)})$ brings a metric $\tilde\theta_a$ on $K$ that is H\"older equivalent to the original $\rho$, which is similar to the previous investigation;  also, note that the index map $\Phi$ on $(X({\bf s}),E^{(b)})$ rarely satisfies the separation conditions in Section \ref{sec:5}, and hence we cannot expect that $(X({\bf s}),E^{(b)})$ has bounded degree.
The more interesting question is to investigate the hyperbolicity of $AI_\infty$-graph $(X({\bf s}),E^{(\infty)})$ without assuming  any separation property. The problem is difficult in general. However, for post critically finite (p.c.f.) sets, we have some rather complete conclusions, as well as a new connection to the harmonic structure and resistance networks in analysis on fractals.

\bigskip

We recall the notion of p.c.f.~sets (without assuming self-similarity) \cite{Ki1}.  Let $\Sigma^\infty := \{i_1i_2\cdots: \ i_k \in \Sigma,  \   k\geq 1\}$ be the set of infinite words, and let $\varpi: \Sigma^\infty \to K$ be the natural surjection defined by
$$
\{\varpi(i_1i_2 \cdots)\} = {\bigcap}_{k=1}^\infty S_{i_1i_2\cdots i_k}(K).
$$
The {\it shift map} $\sigma: \Sigma^\infty \to \Sigma^\infty$ is given by $\sigma(i_1i_2i_3\cdots) = i_2i_3i_4\cdots$. Define the {\it critical set} and the {\it post critical set} by
\begin{equation} \label{eq6.3}
\mathcal C:= \varpi^{-1}\Big({\bigcup}_{i,j \in \Sigma, i \neq j}\big(S_i(K) \cap S_j(K)\big)\Big)\quad \hbox{and} \quad \mathcal P := {\bigcup}_{n=1}^\infty \,\sigma^n(\mathcal C)
\end{equation}
respectively. We call the IFS $\{S_j\}_{j=1}^N$ (or $K$) {\it post critically finite} (p.c.f.) if $\mathcal P$ is a finite set. Also we define $V_0 = \varpi(\mathcal P)$ as the {\it boundary} of $K$. A p.c.f.~set $K$ has the property that every two cells $S_i(K), S_j(K)$, $i, j \in \Sigma, i \neq j$ are disjoint or intersect at finitely many points.
An important consequence is that \cite[Proposition 1.3.5]{Ki1}  for any distinct $x, y$ in the same $X_n({\bf s})$,
$$
S_x(K) \cap S_y(K) = S_x(V_0) \cap S_y(V_0).
$$

\medskip

\begin{theorem} \label{th6.3}
Let $\{S_j\}_{j =1}^N$ be a contractive IFS that satisfies the p.c.f.~property. Then there exists an integer $m>0$ such that for any ${\bf s} \in (0,1)^N$, the $AI_\infty$-graph $(X({\bf s}), E^{(\infty)})$ is an $(m,1)$-departing expansive graph, and hence $(X({\bf s}), E^{(\infty)})$ is  admissible. Moreover, it has bounded degree.
\end{theorem}

\begin{proof} It is clear that the $AI_\infty$-graph $(X({\bf s}),E^{(\infty)})$ is expansive.
 We show that $(X({\bf s}),E^{(\infty)})$ is $(m,1)$-departing for some positive integer $m$, i.e.,
$$
d_h(x,y) = 2 \hbox{ with } |x|=|y| \geq m \quad \Rightarrow \quad d_h(x^{(-m)},y^{(-m)}) \leq 1,
$$
where $\mathcal J_{-m}(x) = \{x^{(-m)}\}$.

\vspace{0.1cm}

Set $\ell_*:= \min_{\xi \neq \eta \in V_0} \{\rho(\xi,\eta)\}$, $\ell^*:= \max_{\xi,\eta \in V_0} \{\rho(\xi,\eta)\}$ and $m:=\lfloor \frac{\log (\ell_*/\ell^*)}{\log r^*}\rfloor + 1$, where $r^* = \max_i\{r_i\}$.
Let $[x,z,y]$ be a horizontal geodesic. Suppose the statement is not true. Then $d_h(x^{(-m)},y^{(-m)})=2$, so that $[x^{(-m)},z^{(-m)},y^{(-m)}]$ is also a horizontal geodesic. As $d_h(x, y) =2$, we can choose distinct $\xi_1, \xi_2$ such that $\xi_1 \in S_x(K) \cap S_z(K)$ and $\xi_2 \in S_y(K) \cap S_z(K)$.
 Observe that
$$
S_x(V_0) \cap S_z(V_0) = S_x(K) \cap S_z(K) \subset S_{x^{(-m)}}(K) \cap S_{z^{(-m)}}(K) = S_{x^{(-m)}}(V_0) \cap S_{z^{(-m)}}(V_0).
$$
Thus $\xi_1 \in S_z(V_0) \cap S_{z^{(-m)}}(V_0)$, and so does $\xi_2$.
Let $w \in \Sigma^*$ satisfy $z=z^{(-m)}w$.
It follows that
$$
\xi_i \in S_z(V_0) \cap S_{z^{(-m)}}(V_0) = S_{z^{(-m)}}(V_0 \cap S_w(V_0)), \qquad i=1,2.
$$
Therefore there exists $\eta_i \in V_0$ such that
$S_w(\eta_i) \in V_0$ (and  $S_z(\eta_i) = \xi_i$) for $i=1,2$. The fact that $w \in \bigcup_{k \geq m} \Sigma^k$ implies
$$
\ell_* \leq \rho(S_w(\eta_1),S_w(\eta_2)) \leq (r^*)^m \rho(\eta_1,\eta_2) \leq (r^*)^m\ell^*.
$$
Thus $m \leq \frac{\log (\ell_*/\ell^*)}{\log r^*}$, contradicting the choice of $m$ above. Hence $(X({\bf s}),E^{(\infty)})$ is $(m,1)$-departing, and is an admissible augmented tree (Theorem \ref{th2.11} and  Proposition \ref{th4.6}).

\vspace{1mm}

Finally, we prove the bounded degree property. Set $s^*:=\max_{j \in \Sigma}\{s_j\}$. For  $x \in X_n({\bf s})$ and $xv \in X_{n+1}({\bf s})$, using \eqref{eq6.2} we have
$$
s_v = s_{xv} \cdot s_x^{-1} \geq s_*^{n+2} \cdot s_*^{-n} = s_*^2.
$$
Therefore $v \in \bigcup_{k=1}^{m'} \Sigma^k$ where $m' := \lfloor\frac{2\log s_*}{\log s^*}\rfloor$, and hence $\#\mathcal J_1(x) \leq \#(\bigcup_{k=1}^{m'} \Sigma^k) < N^{m'+1}$.  Using \eqref{eq6.3}, it follows that $\sup_{\xi \in K} \{\#(\varpi^{-1}(\xi))\} \leq \# \mathcal C < \infty$. For $x,y \in X({\bf s})$, note that $S_x(K) \cap S_y(K) \subset S_x(V_0)$. Therefore
$$
\#\{y \in X:(x,y) \in E_h\} \leq \#(\varpi^{-1}(S_x(V_0))) \leq \#V_0 \cdot {\sup}_{\xi \in K} \{\#(\varpi^{-1}(\xi))\} \leq \#V_0 \cdot \#\mathcal C.
$$
As $\deg(x) \leq 1+ \#\mathcal J_1(x) + \#\{y \in X:(x,y) \in E_h\}$ for all $x \in X$, the graph $(X({\bf s}),E^{(\infty)})$ has bounded degree.
\end{proof}

\bigskip

Consequently,  the map $\kappa : \partial X({\bf s})\to K$ is a bijection as in Definition \ref{de4.2} (actually  $\kappa$ is a homeomorphism by Proposition \ref {th4.3}). Hence the Gromov metric $\theta_a$ on $\partial X$ induces a new metric $\tilde\theta_a$ on $K$ by \eqref{eq4.9}.  Next we show that the index map $\Phi$ over $(K,\tilde\theta_a)$ satisfies the condition $(B_a)$ in Definition \ref{de5.1}.

\medskip

\begin{proposition} \label{th6.4}
Let $\{S_j\}_{j =1}^N$ be a contractive p.c.f.~IFS. For ${\bf s} \in (0,1)^N$, let $\tilde\theta_a$ be the metric on $K$ induced by $(X({\bf s}),E^{(\infty)})$. Then there exists $c_0>0$ such that for any $x \in X({\bf s})$, $S_x(K)$ contains a ball of radius $c_0 e^{-a|x|}$ in $(K,\tilde\theta_a)$ (i.e., condition $(B_a)$).
\end{proposition}

\begin{proof}
Note that $(X({\bf s}),E^{(\infty)})$ is $(m,1)$-departing (Theorem \ref{th6.3}). By Theorem \ref{th4.8},
there exists $\gamma>0$ such that for $x,y \in X({\bf s})$,
\begin{equation} \label{eq6.4}
|x|=|y| \ \hbox{ and } \ S_x(K) \cap S_y(K) = \emptyset \ \ \Rightarrow  \ \ {\rm dist}_{\tilde\theta_a}(S_x(K),S_y(K)) > \gamma  e^{-a|x|}.
\end{equation}
Let $\ell := \lfloor\frac{\log (\# \mathcal C)}{\log N}\rfloor+1$, where $\mathcal C$ is the critical set. Then for $x \in X$, there is $y \in \mathcal J_\ell(x)$ such that $S_{y}(K) \subset S_x(K \setminus V_0)$. Choose $\iota(x) \in S_{y}(K)$ arbitrarily. It follows from \eqref{eq6.4} that
$$
{\inf}_{\eta \in K \setminus S_x(K)}\{\tilde\theta_a(\iota(x),\eta)\} \geq {\rm dist}_{\tilde\theta_a} \big (S_y(K), {\cup}_{z \in X_{|x|} \setminus \{x\}}S_z(K)\big ) > \gamma \cdot e^{-a(|x|+\ell)}.
$$
Hence $B_{\tilde\theta_a}(\iota(x), c_0e^{-a|x|}) \subset S_x(K)$ with $c_0 = \gamma e^{-a\ell}$. This completes the proof.
\end{proof}

\bigskip

Let $\alpha = \alpha({\bf s})$ be the positive number such that $\sum_{j \in \Sigma} s_j^\alpha = 1$, and let $\mu_{\bf s}$ be the {\it self-similar measure} with respect to the vector of probability weights $(s_1^\alpha, s_2^\alpha,\cdots, s_N^\alpha)$, i.e., the unique regular Borel probability measure on $K$ that satisfies
\begin{equation} \label{eq6.5}
\mu_{\bf s}(\cdot) = {\sum}_{j \in \Sigma} \, s_j^\alpha \cdot \mu_{\bf s} (S_i^{-1}(\cdot)).
\end{equation}
In particular if the IFS is p.c.f., then $\mu_{\bf s}(S_x(K)) = s_x^\alpha$ for all $x \in X$.

\medskip

\begin{proposition} \label{th6.5}
Let $\{S_j\}_{j =1}^N$ be a contractive IFS that satisfies the p.c.f.~property. For ${\bf s} \in (0,1)^N$, the self-similar measure $\mu_{\bf s}$ is Ahlfors-regular with exponent $(-\alpha \log s_* / a)$ on $(K,\tilde\theta_a)$, i.e.,
\begin{equation} \label{eq6.6}
\mu_{\bf s}(B_{\tilde\theta_a}(\xi,r)) \asymp r^{-\alpha \log s_* / a}, \qquad \forall \ \xi \in K,\ r \in (0,1).
\end{equation}
\end{proposition}

\begin{proof}
Consider the $AI_\infty$-graph $(X({\bf s}), E^\infty)$. For $x \in X({\bf s})$, set
$$
\Phi^1(x) := \bigcup\{\Phi(y): d_h(x,y) \leq 1\} = \kappa(\mathcal J^1_\partial(x)),
$$
where the index map $\Phi(x) = S_x(K)$.
By Proposition \ref{th3.3} and Theorem \ref{th6.3}, there is $C_0 \geq 1$ such that
\begin{equation*}
B_{\tilde\theta_a}(\xi,C_0^{-1}e^{-a|x|}) \subset \Phi^1(x) \subset B_{\tilde\theta_a}(\xi,C_0e^{-a|x|}), \qquad \forall\ x \in X({\bf s}), \ \xi \in \Phi(x).
\end{equation*}
Using $s_*^{\alpha(|x|+1)} < s_x^\alpha = \mu_{\bf s}(\Phi(x)) \leq \mu_{\bf s}(\Phi^1(x)) \leq t s_*^{\alpha|x|}$ where $t :=\sup_{x \in X({\bf s})} \deg(x) $($<\infty$ by Theorem \ref{th6.3}), we have
\begin{equation*}
\begin{cases}
\ \mu_{\bf s} (B_{\tilde\theta_a}(\xi,C_0^{-1}e^{-an})) \leq t s_*^{\alpha n}, \\
\ \mu_{\bf s} (B_{\tilde\theta_a}(\xi,C_0e^{-an})) \geq s_*^{\alpha(n+1)},
\end{cases}
\qquad \forall \ \xi \in K,\  n \geq 0.
\end{equation*}
This proves \eqref{eq6.6}.
\end{proof}

\bigskip

In the following, we  show that the metric measure space $(K,\tilde\theta_a,\mu_{\bf s})$  plays a special role  in connection with the study of  local regular Dirichlet forms (which give a Laplacian) on $K$ with a regular harmonic structure. This will also extend a consideration by Hu and Wang\cite{HW}, in which they studied the relation between the resistance metric $R$ and the Euclidean metric for IFS on $\mathbb R^d$.

\medskip

We first recall some notations. A {\it (discrete) Laplacian} $H = [H_{pq}]_{p,q \in V_0}$ on $V_0$ is a non-positive definite matrix on $V_0$ that satisfies $\sum_{q \in V_0} H_{pq} = 0$ for all $p \in V_0$, and $H_{pq} \geq 0$ for all distinct $p,q \in V_0$.  For a weight vector ${\bf s } \in (0,1)^N$, let $V_n = V_n({\bf s}) :={\bigcup}_{x \in X_n({\bf s})} S_x(V_0)$ for $n \geq 1$, and $V_* = V_*({\bf s}) := \bigcup_{n=0}^{\infty} V_n({\bf s})$. Denote the collection of real-valued functions on $V_n$ (or $V_*$) by $\ell(V_n)$ (or $\ell(V_*)$ respectively). For a Laplacian $H$ on $V_0$, we have
$$
(u, Hv) = {\sum}_{p\in V_0} u(p) \big ({\sum}_{q \in V_0} H_{pq} \, v(q)\big ), \quad u, v \in \ell(V_0).
$$
We define the energy form ${\mathcal E}_n$ on $V_n$ by
$$
{\mathcal E}_0 [u] = - (u, Hu), \qquad {\mathcal E}_n [u] =  {\sum}_{x \in X_n({\bf s})} s_x^{-1} {\mathcal E}_0[u\circ S_x], \quad \hbox{for} \  u \in \ell(V_n), \  \ n\geq 1
$$
(i.e.,
$\mathcal E_n[u] :=  \frac 12 \, {\sum}_{x \in X_n({\bf s})} s_x^{-1} {\sum}_{p,q \in V_0} H_{pq}|u(S_x(p))-u(S_x(q))|^2$).
We say that the pair $(H,{\bf s})$ is a {\it regular harmonic structure} \cite{Ki1} of $K$ if
\begin{equation} \label{eq6.7}
\min\{\mathcal E_{n+1}[v]: v \in \ell(V_{n+1}), \ v=u  \hbox{ on } V_n\} = \mathcal E_n[u], \qquad \forall\ n \geq 0 , \  u \in \ell(V_n).
\end{equation}
This implies that for $u \in \ell(V_*)$, $\{\mathcal E_{n}[u]\}_{n=0}^\infty$ is an increasing sequence  (here in each $\mathcal E_n[u]$, $u$ is restricted on $V_n$), and
\begin{equation} \label{eq6.8}
\mathcal E[u]:= {\lim}_{n \to \infty}\, \mathcal E_n[u] = {\sup}_{n \geq 1}\,\mathcal E_n[u], \qquad u \in \ell(V_*).
\end{equation}
The $u \in \ell(V_*)$ can be extended continuously to a function on $K$ if $\mathcal E[u]< \infty$. This defines
the local regular Dirichlet form $(\mathcal E, \mathcal D)$  with $\mathcal D := \{u \in C(K): \mathcal E[u] < \infty\}$, where $C(K)$ denotes the space of continuous functions on $K$.
This energy form is {\it self-similar} in the sense that for any $n \geq 1$,
\begin{equation} \label{eq6.9}
\mathcal E [u] = {\sum}_{x \in X_n({\bf s})} s_x^{-1} \mathcal E[u \circ S_x], \qquad \forall \ u \in \mathcal D.
\end{equation}

\bigskip

Define the {\it effective resistance} between two nonempty compact subsets $F,G \subset K$ by
\begin{equation*}
R(F,G) = (\inf\{\mathcal E[u]: u = 1 \hbox{ on } F, \hbox{ and } u = 0 \hbox{ on } G\})^{-1}
\end{equation*}
if $F$ and $G$ are disjoint, and $=0$ otherwise.


\medskip

\begin{lemma} \label{th6.6}
Let $K$ be a connected p.c.f.~set that possesses a regular harmonic structure $(H,{\bf s})$. Then there exists $\gamma'>0$ such that for any $x,y \in X({\bf s})$ with $|x|=|y|$, the inequality
$$
R(S_x(K),S_y(K)) \geq \gamma' \cdot s_*^{|x|}
$$
holds whenever $S_x(K) \cap S_y(K) = \emptyset$.
\end{lemma}

\begin{proof}
Let $x,y \in X_n{(\bf s)}$ with $S_x(K) \cap S_y(K) = \emptyset$. The p.c.f.~property implies
\begin{equation} \label{eq6.10}
R(S_x(K),S_y(K)) = R(S_x(V_0), S_y(V_0)) \geq R(S_x(V_0), V_n\setminus S_x(V_0)).
\end{equation}
 By \eqref{eq6.7} and \eqref{eq6.8}, there exists a function $u$ on $K$ such that $u = 1$ on $S_x(V_0)$, $u=0$ on $V_n \setminus S_x(V_0)$, and
\begin{equation} \label{eq6.11}
\mathcal E_n[u] = \mathcal E[u] = R(S_x(V_0), V_n\setminus S_x(V_0))^{-1}.
\end{equation}
Let $\mathfrak H(x):= \{z \in X_n({\bf s}) \setminus \{x\}: S_z(K) \cap S_x(K) \neq \emptyset\}$. Using the bounded degree property of the $AI_\infty$-graph (Theorem \ref{th6.3}), we have
$$
\#\mathfrak H(x)= \#\{z \in X({\bf s}): (z,x) \in E_h\} < {\sup}_{x \in X({\bf s})}\{\deg(x)\}=:t < \infty.
$$
It follows that
\begin{align*}
\mathcal E_n[u] = \frac 12\,{\sum}_{z \in \mathfrak H(x)}\, s_z^{-1} {\sum}_{p,q \in V_0}\, H_{pq}|u(S_z(p))-u(S_z(q))|^2 \leq \gamma'^{-1} s_*^{-n},
\end{align*}
where $\gamma'^{-1}:= t s_*^{-1} \big (\frac{\# V_0}{2} \big )^2 \max\{H_{pq}: p \neq q \in V_0\} $. This together with \eqref{eq6.10} and \eqref{eq6.11} proves the lemma.
\end{proof}

\bigskip

For $\xi,\eta \in K$, we write $R(\xi,\eta)$ instead of $R(\{\xi\},\{\eta\})$ for short. It is well-known that $R(\cdot,\cdot)$ is a metric on $K$, called the {\it resistance metric}. By definition, for compact subsets $F,G \subset K$ we have
\begin{equation} \label{eq6.12}
{\rm dist}_R(F,G) := \inf\{R(\xi,\eta): \xi \in F,\, \eta \in G\} \geq R(F,G).
\end{equation}
As a consequence of Lemma \ref{th6.6}, we see that the metric space $(K,R)$ with the index map $\Phi$ has the property in \eqref{eq6.4}.

\medskip

\begin{theorem} \label{th6.7}
Let $K$ be a connected p.c.f.~set that possesses a regular harmonic structure $(H,{\bf s})$. Then the metric $\tilde\theta_a$ induced by $(X({\bf s}),E^{(\infty)})$ satisfies
\begin{equation} \label{eq6.13}
\tilde\theta_a(\xi,\eta) \asymp R(\xi,\eta)^{-a/\log s_*} , \qquad \forall \ \xi,\eta \in K.
\end{equation}
\end{theorem}

\begin{proof}
Firstly we prove that the index map $\Phi$ is of exponential type-$(b)$ under $R$, where $b:=|\log s_*|$. For this, let $x \in X_n({\bf s})$ and $\xi,\eta \in S_x(K)$. Then for $u \in \mathcal D$,
\begin{align*}
 |u(\xi)-u(\eta)|^2 =\ & |u(S_x(\xi'))-u(S_x(\eta'))|^2 \qquad  \hbox{(here $S_x(\xi')=\xi$ and $S_x(\eta')=\eta$)} \\
\leq\ & R(\xi',\eta')\ \mathcal E[u \circ S_x] \ \leq \  |K|_R \ \mathcal E[u \circ S_x] \\
\leq\ & |K|_R \cdot s_x \mathcal E[u] \ \leq \ |K|_R \cdot  s_*^n \mathcal E[u],
\end{align*}
where the diameter $|K|_R < \infty$ (cf.~\cite[Theorem 3.3.4]{Ki1}), and the third inequality follows from the the energy self-similar identity  \eqref{eq6.9}. Therefore, by using an equivalent expression of the effective resistance \cite{Ki1},
$$
R(\xi,\eta) = \sup \{\frac{|u(\xi)-u(\eta)|^2}{\mathcal E[u]}: u \in \mathcal D, \ \mathcal E[u] \neq 0\} \leq |K|_R \cdot s_*^n = |K|_R \cdot e^{-bn}.
$$
This proves that $|\Phi(x)|_R \leq |K|_R \cdot e^{-b|x|}$ for all $x \in X$.

\vspace{1mm}

By \eqref{eq6.12} and Lemma \ref{th6.6}, the index map $\Phi$ satisfies \eqref{eq4.11} with $k=1$ under $R$.
It follows from Theorem \ref{th4.8} that the  bijection $\kappa: (\partial X, \theta_a) \to (K,R)$ satisfies \eqref{eq4.10} with $\rho=R$. From the definition \eqref{eq4.9} of $\tilde\theta_a$, we see that $\tilde\theta_a(\cdot,\cdot) \asymp R(\cdot,\cdot)^{a/b}$ on $K$.
\end{proof}

\section{Appendix: IFSs and augmented trees}
\label{sec:ex}

In this Appendix, for the convenience of the reader, we summarize some notations and known facts on iterated function systems, as well as some background of this paper.

\medskip

Let  $(M,\rho)$ be a complete metric space, and let $\{S_j\}_{j=1}^N$ ($N \geq 2$) be a {\it contractive iterated function system} (IFS) on $(M,\rho)$ \cite{Ki1}, i.e., each $S_j: M \to M$ satisfies
\begin{equation} \label{eq7.1}
r_j := \sup\{\frac{\rho(S_j(\xi),S_j(\eta))}{\rho(\xi,\eta)}: \xi, \eta \in M, \, \xi \neq \eta\}\ < \ 1.
\end{equation}
 Then there exists a unique nonempty compact set $K \subset M$ satisfying
$
K = {\bigcup}_{j=1}^N \, S_j(K),
$
called the {\it attractor} of $\{S_j\}_{j=1}^N$;  $K$  is called a {\it self-similar set} if $M= {\mathbb R}^n$ and the $S_j$'s are similitudes, i.e., $|S_j(\xi) -S_j(\eta)| = r_j |\xi-\eta|$.

\medskip

Let the {\it alphabet set} $\Sigma := \{1,2, \cdots, N\}$. Write $\Sigma^0 := \{\vartheta\}$ ($\vartheta$ is the {\it empty word}), and for $n \geq 1$, $\Sigma^n := \{x = i_1 \cdots i_k\cdots i_n: i_k \in \Sigma, \, \forall\ k\}$. Let $\Sigma^*:=\bigcup_{n=0}^\infty \Sigma^n$ denote the {\it symbolic space} of finite words.  This $\Sigma^*$ has a natural $N$-ary tree structure  with the root $\vartheta$.
For $x = i_1 \cdots i_n \in X$, write $r_x :=  r_{i_1} r_{i_2} \cdots r_{i_n}$, $S_x := S_{i_1} \circ S_{i_2} \circ \cdots \circ S_{i_n}$ and $K_x := S_x(K)$ for short.

\vspace{0.1cm}

 For a self-similar set $K$ of a {\it homogeneous}  IFS $\{S_j\}_{j=1}^N$ (i.e., $r_j = r$ for all $j$), coding the iterations by the tree of symbolic space $\Sigma^*$ is natural, as each $K_x$ with $x \in \Sigma^n$ has a constant diameter $r^n|K|$.
But in a non-homogeneous case, the diameters of the cells on each level are not comparable. A common way is to regroup the indices  as follows: let $r_* = \min_{j \in \Sigma}\{r_j\}$, $X_0 = \{\vartheta\}$,
\begin{equation} \label{eq7.2}
 X_n = \{ x= i_1i_2\cdots i_k \in \Sigma^*: r_x \leq r_*^n< r_{i_1} r_{i_2}\cdots r_{i_{k-1}}\}, \quad n \geq 1,
\end{equation} and $X= \bigcup_{n=0}^\infty X_n$. This $X$ has a natural tree structure $E_v$, and the diameters of the cells in $\{K_x\}_{x \in X_n}$ are comparable with $r_*^n$.

 \medskip
A contractive IFS  of similitudes is said to satisfy the {\it open set condition} (OSC) if there is a bounded nonempty  open set $U$ such that $S_j (U) \subset  U$ for all $j \in \Sigma$,  and $S_i(U) \cap S_j(U) = \emptyset$ for all $i \not = j$.  The OSC is one of the most fundamental conditions in fractal geometry.  For a self-similar set $K$, the OSC yields an explicit expression of the Hausdorff dimension $s$ of $K$ by $\sum_{j=1}^N r^s_j =1$, and $K$ supports the $s$-Hausdorff measure.
Furthermore, the OSC is equivalent to the following property (cf.~\cite[Theorem 2.2]{Sc}): for any $c>0$, there is a constant $\ell = \ell(c)$ such that
\begin{equation} \label{eq7.3}
\forall \ \eta \in K \ \hbox{and integer} \ n>0, \  B(\eta, cr_*^n) \cap K_x \not = \emptyset \ \hbox {  for at most } \ell \hbox { of }  x \in X_n.
\end{equation}

\medskip
We augment the tree $(X,E_v)$ by a set of horizonal edges
\begin{equation} \label{eq7.4}
E_h = {\bigcup}_{n = 1}^\infty \{ (x, y) \in X_n \times X_n \setminus \Delta: K_x \cap K_y \not = \emptyset\},
\end{equation}
and let $E= E_v\cup E_h$ (cf.~\cite{Ka, LW1}).  It has been  proved that

\medskip
\begin{theorem} \label{th7.1} \hspace{-2mm} {\rm \cite {LW1}}
Let $\{S_j\}_{j=1}^N$ be an IFS of contractive similitudes that satisfies the
OSC. Then $(X, E)$ is a hyperbolic graph, and for the hyperbolic boundary $\partial X$, the canonical identification $\kappa :\partial X \to K$ is a homeomorphism. Furthermore, $\kappa$ is a H\"older equivalence if the IFS satisfies the condition (H), i.e.,
there exists a constant $c>0$ such that
\begin{equation} \label{eqH}
 \forall  \ n>0, \ x,y \in X_n, \ K_x \cap K_y = \emptyset \ \Rightarrow \ {\rm dist}(K_x, K_y) \geq c r_*^n.
\end{equation}
\end{theorem}

\medskip

The identification of $K$ and $\partial X$ has been applied to study the Lipschitz equivalence of self-similar sets {\cite {LL, DLL}}. If we enlarge the horizontal edge set to be
\begin{equation} \label{eq7.5}
E'_h = {\bigcup}_{n=1}^\infty \{ (x, y) \in X_n \times X_n \setminus \Delta:  {\rm dist} (K_x, K_y) \leq \gamma r_*^n\}
\end{equation}
for some $\gamma>0$ and let $E'= E_v \cup E'_h$ (cf.~\cite{LW3}), then we can improve Theorem \ref {th7.1} as

\begin {theorem}  \label{th7.2} \hspace{-2mm} {\rm \cite{LW3}}
For any IFS $\{S_j\}_{j=1}^N$ of similitudes, $(X, E')$ is a hyperbolic graph, and
the canonical identification $\iota : \partial X \to K$ is a H\"older equivalence. Moreover, $(X,E')$ has bounded degree if and only if $\{S_j\}_{j=1}^N$ satisfies the OSC.
\end{theorem}

\bigskip

In \cite {KLW1}, thanks to the bounded degree property, we can introduce a class of transient reversible random walks on $(X,E')$ such that the Martin boundary, $\partial X$ and $K$ are homeomorphic, by which we  obtain a jump kernel (i.e., Na\"im  kernel) to study the  induced energy form on $K$.

\bigskip

For  IFS $\{S_j\}_{j=1}^N$ with overlaps, it is possible that $S_x = S_y$ for some different $x,y \in X$, where $X$ is defined in \eqref{eq7.2}. For example, let $M=\mathbb R$, $S_1(x) = r x$ and $S_2(x) = r x + (1-r)$, where $r = \frac{\sqrt{5} -1}{2}$ is the golden ratio. Then $S_{122} = S_{211}$ (see also Example \ref{ex2.2}).
In this case, we can define an equivalence relation $\simeq$ on $X$ by $x \simeq y$ if and only if $S_x = S_y$. Then there is a natural vertical graph $(X^\sim, E_v^\sim)$ as the quotient of $(X,E_v)$ with respect to $\simeq$, which is not a tree unless the relation $\simeq$ is trivial \cite{LW3}.  It was proved in \cite{Wa} that the associated augmented tree of \eqref {eq7.4} in $(X^\sim, E^\sim)$ is hyperbolic if the self-similar set $K$ has positive Lebesgue measure, or
the IFS $\{S_j\}_{j=1}^N$ satisfies the {\it weak separation condition} (WSC) (cf.~\cite {LN},\cite[Theorem 2.1]{LW}), i.e., the condition \eqref{eq7.3} with $X_n$ replaced by the quotient $X_n /\simeq$.
Some more variants were discussed in \cite {Wa}.

\bigskip
\bigskip

\noindent {\bf Acknowledgements}: The authors would like to thank Professor Alexander Grigor'yan for many valuable discussions. They also like to extend their thanks to Doctor Leung-Fu Cheung, Professors Qingsong Gu and Sze-Man Ngai for going through the manuscript and making some suggestions.

\bigskip
 {\small\bibliographystyle{amsplain}

}

\bigskip
\bigskip

\noindent Shi-Lei Kong, Fakult{\"a}t f{\"u}r Mathematik, Universit{\"a}t Bielefeld, Postfach 100131, 33501 Bielefeld, Germany. \\
skong@math.uni-bielefeld.de

\bigskip

\noindent Ka-Sing Lau, Department of Mathematics, The Chinese University of Hong Kong, Hong Kong.\\
\& Department of Mathematics, University of Pittsburgh, Pittsburgh, PA 15260, USA. \\
kslau@math.cuhk.edu.hk

\bigskip

\noindent Xiang-Yang Wang, School of Mathematics, Sun Yat-Sen University, Guangzhou, China.\\
mcswxy@mail.sysu.edu.cn

\end{document}